\documentclass[12pt,reqno]{amsart}
\usepackage[margin=1in]{geometry}
\usepackage{amsmath,amssymb,amsthm,graphicx,amsxtra, setspace}
\usepackage[utf8]{inputenc}
\usepackage{mathrsfs}
\usepackage{hyperref}
\usepackage{upgreek}
\usepackage{mathtools}
\usepackage[mathcal]{euscript}
\usepackage{xcolor}
\allowdisplaybreaks

\usepackage[pagewise]{lineno}

\usepackage{graphicx,eurosym}
\usepackage{hyperref}
\usepackage{mathtools}

\usepackage[cyr]{aeguill}

\colorlet{darkblue}{blue!50!black}

\hypersetup{
	colorlinks,%
	citecolor=blue,%
	filecolor=red,%
	linkcolor=darkblue,%
	urlcolor=blue,%
	pdfnewwindow=true,%
	pdfstartview={FitH}
}

\colorlet{darkblue}{red!100!black}

\newtheorem{theorem}{Theorem}[section]
\newtheorem{lemma}[theorem]{Lemma}
\newtheorem{proposition}[theorem]{Proposition}

\newtheorem{definition}[theorem]{Definition}
\newtheorem{example}[theorem]{Example}
\newtheorem{remark}[theorem]{Remark}

\newtheorem{hypothesis}[theorem]{Hypothesis}

\let\originalleft\left
\let\originalright\right
\renewcommand{\left}{\mathopen{}\mathclose\bgroup\originalleft}
\renewcommand{\right}{\aftergroup\egroup\originalright}

\renewcommand{\d}{\/\mathrm{d}\/}

\def\w{\textbf{W}^{\varepsilon}_{{\theta}^{\varepsilon}}}

\def\L{\mathbb{L}}
\def\A{\mathrm{A}}
\def\I{\mathrm{I}}

\def\C{\mathrm{C}}
\def\f{\boldsymbol{f}}

\def\B{\mathrm{B}}
\def\D{\mathrm{D}}
\def\y{\boldsymbol{y}}

\def\E{\mathbb{E}}
\def\X{\mathbb{X}}
\def\x{\boldsymbol{x}}

\def\z{z }
\def\v{\boldsymbol{u}}
\def\w{\boldsymbol{w}}
\def\W{\mathrm{W}}

\def\N{\mathbb{N}}

\def\V{\mathbb{V}}
\def\wi{\widetilde}

\def\u{\mathrm{U}}

\def\u{\boldsymbol{v}}
\def\H{\mathbb{H}}

\newcommand{\R}{\mathbb{R}}

\renewcommand{\d}{\/\mathrm{d}\/}

\newcommand{\Addresses}{{
		\footnote{
			\noindent \textsuperscript{1,2}Department of Mathematics, Indian Institute of Technology Roorkee-IIT Roorkee,
			Haridwar Highway, Roorkee, Uttarakhand 247667, INDIA.\par\nopagebreak
			\noindent  \textit{e-mail:} \texttt{Manil T. Mohan: maniltmohan@ma.iitr.ac.in, maniltmohan@gmail.com.}
			
			\textit{e-mail:} \texttt{Kush Kinra: kkinra@ma.iitr.ac.in.}
			
			\noindent \textsuperscript{*}Corresponding author.
			
			\textit{Key words:}  Bi-spatial pullback random attractor, 2D stochastic Navier-Stokes equations, whole space, exponential stability, invariant measures.
			
			Mathematics Subject Classification (2020): Primary 35B41, 35Q35; Secondary 37L55, 37N10, 35R60.

}}}

\begin{document}
	
	\title[2D SNSE on the whole space]{Bi-spatial random attractors, a stochastic Liouville type theorem and ergodicity for stochastic Navier-Stokes equations on the whole space
		\Addresses}
	
	\author[K. Kinra and M. T. Mohan]
	{Kush Kinra\textsuperscript{1} and Manil T. Mohan\textsuperscript{2*}}

	\maketitle
	
	\begin{abstract}
		This article concerns the random dynamics and asymptotic analysis of the well known mathematical model,
		\begin{align*}
			\frac{\partial \boldsymbol{v}}{\partial t}-\nu \Delta\boldsymbol{v}+(\boldsymbol{v}\cdot\nabla)\boldsymbol{v}+\nabla p=\boldsymbol{f}, \ \nabla\cdot\boldsymbol{v}=0,
		\end{align*}
		 the Navier-Stokes equations. We consider the two-dimensional  stochastic Navier-Stokes equations (SNSE) driven by a \textsl{linear multiplicative white noise of It\^o type} on the whole space $\mathbb{R}^2$. Firstly, we prove that  the non-autonomous 2D SNSE generates a bi-spatial $(\mathbb{L}^2(\mathbb{R}^2),\mathbb{H}^1(\mathbb{R}^2))$-continuous random cocycle. Due to the bi-spatial continuity property of the random cocycle associated with SNSE, we show that if the initial data is in $\mathbb{L}^2(\mathbb{R}^2)$, then there exists a unique bi-spatial $(\mathbb{L}^2(\mathbb{R}^2),\mathbb{H}^1(\mathbb{R}^2))$-pullback random attractor for non-autonomous SNSE which is compact and attracting not only in $\mathbb{L}^2$-norm but also in $\mathbb{H}^1$-norm. Next, as a consequence of the existence of pullback random attractors, we prove the existence of a family of invariant sample measures for non-autonomous random dynamical system generated by 2D non-autonomous SNSE. Moreover, we show that the family of invariant sample measures satisfies a stochastic Liouville type theorem. Finally, we discuss the existence of an invariant measure for the random cocycle associated with 2D autonomous SNSE. We prove the uniqueness of invariant measures for $\boldsymbol{f}=\mathbf{0}$ and for any $\nu>0$ by using the linear multiplicative structure of the noise coefficient and exponential stability of solutions. The above results for SNSE defined on $\mathbb{R}^2$ are totally new, especially the results on bi-spatial random attractors and stochastic Liouville type theorem for 2D SNSE with linear multiplicative noise are obtained in any kind of domains for the first time. We observe that in contrast to Stratonovich noise, which is used widely in the literature to study the random dynamics of SNSE, It\^o noise is more adequate in the case of whole space. This work settles down several open problems regarding random attractors, invariant measures and ergodicity for 2D SNSE on the whole space. 
	\end{abstract}

	\section{Introduction} \label{sec1}\setcounter{equation}{0}
	\subsection{Literature survey and motivations}
		It is well known that the explicit solutions of ordinary/partial differential equations (ODE/PDE) are very difficult to find (cf. \cite{KH,SSTC} and references therein). Therefore, one requires  a qualitative theory to understand the asymptotic behavior of their solutions. The theory of attractors plays an important role to capture the long time behavior of the solutions of ODE/PDE. The theory of global attractors for the deterministic infinite-dimensional dynamical systems (DS) is well-investigated in \cite{JCR,R.Temam}, etc. The theory of global attractors (for deterministic DS) has been extended to random attractors (for compact random DS, cf. \cite{Arnold}) in the works \cite{CDF,CF} etc., and it has been used for several physically relevant stochastic models (cf. \cite{FY,KM3,KRM,XC} etc. and references therein). The author in \cite{SandN_Wang}  introduced a \emph{sufficient and necessary criteria for the existence of a unique pullback random attractor} for non-compact non-autonomous random DS and applied it to stochastic reaction-diffusion equations driven by additive noise. Later, it has been applied to several stochastic models, cf. \cite{GGW,KM6,PeriodicWang,RKM,WSLW},  etc., and references therein.

		   The current work is mainly focused on the random dynamics of \emph{autonomous and non-autonomous 2D stochastic Navier-Stokes equations (SNSE) driven by a linear multiplicative white noise} defined on the whole space. Given $\tau\in\R$, we consider the following non-autonomous 2D SNSE on $\R^2$:
	\begin{equation}\label{1}
		\left\{
		\begin{aligned}
			 \d\u&=-\left[-\nu \Delta\u+(\u\cdot\nabla)\u+\nabla p-\boldsymbol{f}\right]\d t+\sigma\u\d \W, \hspace{10mm}   \text{ in }\  \R^2\times(\tau,\infty), \\ \nabla\cdot\u&=0, \hspace{84.5mm} \text{ in } \R^2\times(\tau,\infty), \\
			\u|_{t=\tau}&=\u_{\tau}, \hspace{83.5mm}  x\in \R^2 \ \text{ and }\ \tau\in\R,\\
			\u(x)&\to 0\hspace{85.3mm}   \text{ as }\ |x|\to\infty, 
		\end{aligned}
		\right.
	\end{equation}
	where $\u(x,t)\in \R^2$ stands for the velocity field and $p(x,t)\in\R$ denotes the pressure field, for all $(x,t)\in\R^2\times(\tau,+\infty)$. The external forcing  $\f$ is either autonomous ($\f(x)\in\R^2,$ for $x\in\R^2$) or non-autonomous ($\f(x,t)\in\R^2,$ for $(x,t)\in\R^2\times(\tau,+\infty)$). The coefficient $\nu>0$ represents the \emph{kinematic viscosity} of the fluid and $\sigma>0$ is known as the  \emph{noise intensity}. Here, the stochastic integral should be understood in the sense of  It\^o and $\W=\W(t,\omega)$ is a one-dimensional two-sided Wiener process defined on some filtered probability space $(\Omega, \mathscr{F}, (\mathscr{F}_{t})_{t\in\R}, \mathbb{P})$ (see Section \ref{sec2}).

	 The solvability of 2D deterministic/stochastic Navier-Stokes equations (NSE) on bounded or unbounded domains is well studied in the literature (existence as well as uniqueness), cf. \cite{FSX,FMRT,GZ,MS,SS,R.Temam} etc., and references therein. In three-dimensions, the existence of at least one weak solution of 3D NSE is known due to Leary and Hopf (cf. \cite{Leray,Hopf}). But the uniqueness is still a mathematically challenging open problem. Recently, in \cite{BV}, it has been shown by a convex integration technique that the weak solutions of 3D Navier-Stokes equations are not unique in the class of weak solutions with finite kinetic energy. Therefore, in this work, we focus only on two-dimensions. 
	 
	 The dynamics of 2D deterministic/stochastic NSE including the existence of unique global/ pullback/random attractors also have a plenty of literature, cf. \cite{BL,CLR1,GLS2,GGW,Robinson,Rosa,R.Temam,RKM} etc., and references therein. It has been observed  that the existence of global/pullback/ random attractors for 2D deterministic/stochastic NSE is well-studied either on periodic (cf. \cite{HCPEK,Robinson,R.Temam} etc.), bounded  (\cite{CDF,GGW,KRM} etc.), and unbounded Poincar\'e (\cite{BL,CLR1,GLW,Rosa} etc.) domains, and not on the whole space.  Until now, the existence of random attractors is based on a transformation which converts a stochastic system into an equivalent pathwise deterministic system. Such types of transformations are available in the literature when the noise is either \emph{additive or linear multiplicative}. Moreover, it has been noticed in the literature that  the SNSE driven by a linear multiplicative noise has been considered in the sense of Stratonovich. We consider the SNSE driven by a linear multiplicative noise in the \emph{It\^o sense} which is one of the main motivations  to do this work on the whole space $\R^2$. \emph{The existence of global/pullback attractors for deterministic NSE and global/pullback random attractors for SNSE driven by additive noise on the whole space is still an open problem}.

	  The first aim of this work is to show the existence of a unique bi-spatial pullback random attractor for the non-autonomous SNSE \eqref{1} on the whole space $\R^2$. Stochastic PDE such as 2D SNSE (cf. \cite{MS}), stochastic convective Brinkman-Forchheimer equations (cf. \cite{MTM1}), stochastic fractional power dissipative equation (cf. \cite{WZhao}), stochastic semilinear Laplacian equations (cf. \cite{LGL}) etc., have regularizing solutions. In other words, when initial data is in a Banach space (an initial space), the corresponding solution may belong to a more regular Banach space (a terminal space). To represent the better dynamics for such equations, the authors in \cite{LGL} introduced the concept of \emph{bi-spatial random attractor} which is an invariant and compact random set attracting the subsets of the initial space in the topology of the terminal space. Later, the authors in \cite{CLY} generalized the work of \cite{LGL} for non-autonomous random DS also, which has been applied to several physically relevant models, cf. \cite{rwang2,rwang3}, etc., and references therein. Recently, the author in \cite{WZhao} introduced an abstract theory of bi-spatial pullback attractors for the bi-spatial continuous cocycle and successfully applied it to the stochastic fractional power dissipative equation. To the best of our knowledge, there is no result available in the literature on the \emph{existence of a unique bi-spatial pullback random attractor for 2D SNSE either on Poincar\'e domains (bounded or unbounded) or on  the whole space $\R^2$.}

		The second goal of this paper is to prove the existence of \emph{a family of invariant sample measures} (see Definition \ref{D-ISM} below) which satisfies \emph{a stochastic Liouville type theorem} (cf. \cite{CY}) for the 2D non-autonomous SNSE \eqref{1} on the whole space $\R^2$. This result is obtained with the aid of the existence of pullback random attractors and an abstract theory available in \cite[Theorem 3.1]{CY}. Sufficient criteria for the existence of a family of invariant sample measures for autonomous and non-autonomous random DS (under the assumption on the existence of global/pullback random attractors) is established in \cite{ZWC,CY}, respectively. In addition, the works \cite{ZWC,CY} reveal that the family of invariant sample measures will be supported by the global/pullback random attractors. Note that if the random statistical equilibrium (that is, the instantaneous quantities still vary widely in time but average quantities seem stationary)	has been reached by the 2D SNSE \eqref{1} (or the system \eqref{CNSE} below), then the statistical information do not change with time (cf. \cite{RMSRosa}). In this situation, the shape of the pullback random attractor $\mathscr{A}(t,\omega)$ changes randomly with respect to time along with the sample points $\omega\in\Omega,$ but the measures of $\mathscr{A}(t,\omega)$ on each time are the same with respect to the family of invariant sample measures. This is the stochastic (or random) version of the Liouville Theorem in Statistical Mechanics (cf. \cite{ZWC} for autonomous random DS and  \cite{CY} for non-autonomous random DS). We refer readers to the works \cite{RMSRosa,ZLC} for more detailed information about statistical equilibrium, trajectory statistical solutions and Liouville theorem in Statistical Mechanics (see Remark \ref{Sto-Liv} also). An abstract theory for a \emph{random Liouville type theorem}  has been established in \cite[Theorem 5.1]{CY}. Due to the fact that Brownian paths are continuous but nowhere differentiable, we cannot apply the abstract theory in \cite[Theorem 5.1]{CY} directly to the system \eqref{SNSE} with white noise. Therefore, we prove a stochastic version of  Liouville type theorem and use the terminology \textit{a stochastic Liouville type theorem} (cf. \cite[Subsection 6.3]{CY} and Theorem \ref{SLTT}).

	  The final objective of this article is to demonstrate the \emph{ergodic properties (via proving the existence of a unique invariant measure)} for the 2D autonomous SNSE \eqref{1} which is one of the most considerable qualitative properties of a dynamical system. The ergodic properties of the randomly forced NSE on torus or bounded domains have been extensively studied by several researchers during the past two decades. The first work on the  ergodicity for 2D SNSE on bounded domains (for any $\nu>0$) driven by additive noise (see \cite{Flandoli}, for the existence of invariant measures)  was carried out  in \cite{FM}, where the uniqueness of invariant measures was proved under some extra assumptions on the noise (cf. \cite{FM} for more details). But, in \cite{BF}, the author proved the same results as in \cite{FM} under weaker assumptions on the noise. Later, instead of taking extra assumptions on the noise, the author in \cite{Mattingly} proved the ergodicity for 2D SNSE driven by additive noise for sufficiently large $\nu$. The papers \cite{BF1,HM,HMJC,KS,KS1,MTSS1,MKS,WMS,WM} etc., discussed various results on the ergodic behavior of different types of partial differential equations covering 2D SNSE and  the works \cite{Debussche,Mattingl2}, etc. are good expository articles on this subject, and references therein. The authors in \cite{GMR,KS1}, etc. demonstrated the ergodicity property for 2D SNSE using the asymptotic coupling method, where the proof of uniqueness of invariant measures does not require any restriction on $\nu>0$. In the literature, most of the works regarding the existence and uniqueness of invariant measures as well as ergodic behavior of 2D SNSE have been considered  either on  a torus or on bounded domains. In some of the works such as \cite{BL,BMO} etc., the authors established the existence of invariant measures for 2D SNSE driven by additive or multiplicative It\^o noise  on unbounded Poincar\'e domains, but they have not discussed the uniqueness. Recently, in \cite{Nersesyan}, the author proved the ergodicity results for 2D SNSE   on unbounded Poincar\'e domains, where the author used a random force of additive type having a special structure. It appears to us that the current work is the first one in which \emph{the existence and uniqueness of ergodic and strongly mixing invariant measures  for 2D autonomous SNSE driven by a linear multiplicative noise in the whole space $\R^2$} is discussed for any $\nu>0$ and the results are new in the context of multiplicative noise as well as for 2D SNSE on the whole space $\R^2$.

		\subsection{Difficulties and approaches} 
		If one considers a damped  2D deterministic/stochastic NSE with linear damping, that is,
		\begin{align*}
			\d\u+\left[-\nu \Delta\u+(\u\cdot\nabla)\u+\alpha\u+\nabla p\right]=\boldsymbol{f}\d t+\sigma\u\d \W,  \  \nabla\cdot\u=0, 
		\end{align*}
	where $\alpha>0$ and $\sigma\geq 0$, then	the existence of global/pullback/random attractors can be proved on the whole space (cf. \cite{IPZ,JH} etc.). Even though we do not have any linear damping term in \eqref{1}, by using a suitable transformation (see \eqref{COV} below), we obtain an equivalent pathwise deterministic system (see \eqref{CNSE} below) which helps us to prove the results on the whole space. It is an advantage of considering the stochastic system \eqref{1} in the sense of  It\^o, rather than Stratonovich. 
		
		In order to prove the bi-spatial $(\L^2(\R^2),\H^1(\R^2))$-pullback random attractors for the system \eqref{1}, we need to estimate the nonlinear  functional $\int_{\R^2}(\v(x)\cdot\nabla)\v(x)\cdot\Delta\v(x)\d x$ (see \eqref{ei3} below) in an appropriate way. We observe that one cannot use the estimates obtained in \cite[Chapter 2, section 2.3]{Temam1} to estimate the above nonlinear term. In view of H\"older's, Ladyzhenskaya's  (\cite[Lemma 1, Chapter I]{OAL}), Sobolev's (\cite[Theorem 7, pp. 190]{DMDZ}) and interpolation (\cite[Theorem 6, pp. 200]{DMDZ}) inequalities, we estimate the above nonlinear term (see Remark \ref{TriEsti} below). Note that for the periodic case, this integral is zero \cite[Lemma 3.1]{Temam1}, and for the bounded or unbounded Poincar\'e domains case, one can use the estimates form  \cite[Chapter 2, Section 2.3]{Temam1} to obtain the results similar to this work.  As the existence of pullback attractors for 2D deterministic NSE (\eqref{1} with $\sigma=0$) is still an open problem, the upper semicontinuity (as $\sigma\to 0$) of the obtained pullback random attractors is also an open problem (Remark \ref{rem3.14}). 
		
		The existence of random attractors helps us to obtain the existence of invariant measures (see \cite[Corollary 4.6]{CF} and Definition \ref{IMRDS} below), and a family of invariant sample measures (\cite[Theorem 2.1]{ZWC}, \cite[Theorem 5.1]{CY} and Definition \ref{D-ISM} below) which satisfies a stochastic (or random) Liouville type theorem for the system \eqref{1} (cf. \cite{CF,ZWC} etc.). As discussed above, due to the nowhere differentiability of Brownian paths, we do not apply \cite[Theorem 5.1]{CY} directly. We use the It\^o formula in infinite dimensions along with the invariant property of invariant sample measures to show that the family of invariant sample measures satisfies a stochastic Liouville type equation (see \eqref{SLTT1} in Theorem \ref{SLTT} below). Note that Theorem \ref{SLTT} is different from random Liouville type theorems discussed in \cite[Theorem 3.2]{ZWC} and \cite[Theorem 5.1]{CY}. 
		
		The main difficulties arise in the uniqueness of invariant measures part. We consider the deterministic forcing term $\f=\mathbf{0}$ and obtain the uniqueness of invariant measure (Dirac measure centered at zero) for any $\nu>0,$ where the linear multiplicative structure of the white noise coefficient and exponential stability of solutions (see Lemma \ref{ExpoStability}) play a crucial role. The uniqueness of invariant measures  for $\f\neq \mathbf{0}$ is still an open problem for 2D SNSE in $\R^2$. 
		
		\subsection{Novelties} Most of the results regarding the random dynamics and  asymptotic analysis of 2D SNSE available in the literature are either on bounded domains (cf. \cite{BF,GMR,HM,KRM,RKM}, etc.) or on unbounded Poincar\'e domains (cf. \cite{BL,BMO,Nersesyan}, etc.).   This work settles down several open problems concerning the random dynamics of 2D SNSE defined on the whole space and this work appears to be the first one in this direction. Moreover, we emphasize here that all the results obtained in this paper hold true even  in bounded as well  as unbounded Poincar\'e domains.   The major aims and novelties of this work are:
		\begin{itemize}
			\item [(i)] \emph{For the 2D non-autonomous SNSE \eqref{1}, we prove the existence of unique bi-spatial $(\L^2(\R^2),$ $\H^1(\R^2))$-pullback random attractors (Theorem \ref{MainTheoRA}).}
			
			\item [(ii)] \emph{For the 2D non-autonomous SNSE \eqref{1}, we show the existence of a family of invariant sample measures on $\L^2(\R^2)$ which satisfies a stochastic Liouville type theorem (Theorems \ref{ISM} and  \ref{SLTT}).}
			
			\item [(iii)] \emph{For the 2D autonomous SNSE \eqref{1}, for any $\f\in\L^2(\R^2)$, we establish the existence of an invariant measure in $\L^2(\R^2)$ as well as in $\H^1(\R^2)$ (Theorem \ref{thm6.3} and Remark \ref{4.4}) and its uniqueness for $\f=\mathbf{0}$ (Theorem \ref{UEIM}).}
		\end{itemize}
	The linear multiplicative structure of the It\^o type noise coefficient helps us to resolve the above mentioned problems in $\R^2$. Moreover, we provide some remarks on the existence of $(\L^2(\R^2),$ $\L^2(\R^2))$-pullback random attractors,  the upper semicontinuity of the random attractors with respect to domains, and the asymptotic autonomy of the obtained pullback random attractors as well (Remarks \ref{rem3.13}, \ref{rem315} and \ref{rem3.15}).

\subsection{Outline} In the next section, we provide the necessary function spaces needed for the further analysis, and linear and bilinear operators to obtain an abstract formulation of the system \eqref{1}. Furthermore, we have furnished an abstract formulation of the system \eqref{1}, a pathwise deterministic system \eqref{CNSE} which is equivalent to the system \eqref{1}, random cocycle \eqref{Phi1},  and the universe of tempered sets, in the same section. In section \ref{sec3}, we first provide an abstract result for the existence of a unique bi-spatial pullback random attractor which is adapted from the work \cite[Theorem 2.10]{WZhao} (Theorem \ref{AbstractResult}). In order to apply the abstract result of Theorem \ref{AbstractResult}, we prove that the random cocycle generated by the system \eqref{1} is a bi-spatial $(\L^2(\R^2),\H^1(\R^2))$-continuous cocycle (Lemma \ref{Continuity}). Next, we prove that the random cocycle has a pullback random absorbing set (Lemma \ref{PA1}), the random cocycle is $(\L^2(\R^2),\L^2(\R^2))$-pullback asymptotically compact (Lemma \ref{PA2}) and continuous with respect to the sample points (Lemma \ref{LusinC}). Finally, using an abstract theory (Theorem \ref{AbstractResult}), we establish the main result of this section (Theorem \ref{MainTheoRA}). We start with some basic definitions for invariant measures (adapted from \cite{Arnold}) and prove the existence of a family of invariant sample measures which satisfies a stochastic Liouville type theorem for the system \eqref{1} (Theorems \ref{ISM} and \ref{SLTT}) in section \ref{sec4}. In the final section, we show that there exists an invariant measure for the system \eqref{1} in $\L^2(\R^2)$ as well as in $\H^1(\R^2)$ (Theorem \ref{thm6.3} and Remark \ref{4.4}) for any $\f\in\L^2(\R^2)$. At the end, we prove the exponential stability of solutions of \eqref{1} with $\f=\mathbf{0}$ (Lemma \ref{ExpoStability}) and using that we show the uniqueness of invariant measure for $\f=\mathbf{0}$ (Theorem \ref{UEIM}).

	\section{Mathematical Formulation}\label{sec2}\setcounter{equation}{0}
	In this section, we provide the necessary function spaces and operators needed to obtain the main results of this work. Further, we provide an abstract formulation of the system \eqref{1}, a pathwise deterministic system equivalent to the system \eqref{1} (which helps us to define the non-autonomous random dynamical system) and the universe of tempered sets.
	\subsection{Function spaces} 
	We define the space $\mathcal{V}:=\{\u\in\C_0^{\infty}(\R^2;\mathbb{R}^2):\nabla\cdot\u=0\},$ where $\C_0^{\infty}(\R^2;\mathbb{R}^2)$ denotes the space of all $\mathbb{R}^2$-valued infinite  times differentiable functions with compact support in $\mathbb{R}^2$. Let $\H$ and $\V$ denote the completion of $\mathcal{V}$ in $\L^2(\R^2):=\mathrm{L}^2(\R^2;\mathbb{R}^2)$ and $\H^1(\R^2):=\mathrm{H}^1(\R^2;\mathbb{R}^2)$ norms, respectively. The spaces $\H$ and $\V$ are endowed with the norms $\|\u\|_{\H}^2:=\int_{\R^2}|\u(x)|^2\d x$ and $\|\u\|^2_{\V}:=\int_{\R^2}|\u(x)|^2\d x+\int_{\R^2}|\nabla\u(x)|^2\d x$, respectively. The inner product in the Hilbert space $\H$ is represented by $( \cdot, \cdot)$, and the duality pairing between the spaces $\V$ and $\V^*$ is denoted by $\langle\cdot,\cdot\rangle.$ 	The $m^{\textrm{th}}$-order Sobolev spaces are denoted by $\H^m(\R^2):=\mathrm{H}^m(\R^2;\R^2)$ with the norm $\|\u\|^2_{\H^m(\R^2)}:=\sum\limits_{|\boldsymbol{k}|\leq m}\int_{\R^2}|\D^{\boldsymbol{k}}\u(x)|^2\d x$, where $\boldsymbol{k}=(k_1,k_2)$ is the  multiindex and $|\boldsymbol{k}|=k_1+k_2$. 
	\subsection{Linear operator}\label{LO}
Let $\mathcal{P}: \L^2(\R^2) \to\H$ denote the Helmholtz-Hodge orthogonal projection (cf.  \cite{OAL}). The projection operator $\mathcal{P}:\L^2(\R^2) \to\H$ can be expressed in terms of the Riesz transform (cf. \cite{MTSS}), and the operators $\mathcal{P}$ and $\Delta$ commutes, that is, $\mathcal{P}\Delta=\Delta\mathcal{P}$. Let us define the Stokes operator
\begin{equation*}
	\A\u:=-\mathcal{P}\Delta\u=-\Delta\u,\;\u\in\D(\A):=\V\cap\H^2(\R^2).
\end{equation*}
The operator $\A:\V\to\V^*$ is  linear and continuous. Moreover, the usual norm of  $\H^m(\R^2)$ is equivalent to (cf. \cite[Proposition 1, pp. 169]{DMDZ})
\begin{align}\label{EquiH2}
		\|\u\|_{\H^m(\R^2)}=\left(\int_{\R^2}(1+|\xi|^2)^{m}\left|\widehat{\u}(\xi)\right|^2\d \xi\right)^{\frac{1}{2}}=\|(\I-\Delta)^{\frac{m}{2}}\u\|_{\H}=\|(\I+\A)^{\frac{m}{2}}\u\|^2_{\L^2(\R^2)},
\end{align}
where $\widehat{\u}(\cdot)$ is the Fourier transform of $\u(\cdot)$.
	\subsection{Bilinear operator}
	Let us define the \emph{trilinear form} $b(\cdot,\cdot,\cdot):\V\times\V\times\V\to\R$ by $$b(\v,\u,\w)=\int_{\R^2}(\v(x)\cdot\nabla)\u(x)\cdot\w(x)\d x=\sum_{i,j=1}^2\int_{\R^2}\v_i(x)\frac{\partial \u_j(x)}{\partial x_i}\w_j(x)\d x.$$ An integration by parts gives 
	\begin{equation}\label{b0}
		\left\{
		\begin{aligned}
			b(\v,\u,\v) &= 0,\ \text{ for all }\ \v,\u \in\V,\\
			b(\v,\u,\w) &=  -b(\v,\w,\u),\ \text{ for all }\ \v,\u,\w\in \V.
		\end{aligned}
		\right.\end{equation} If $\v, \u$ are such that the linear map $b(\v, \u, \cdot) $ is continuous on $\V$, the corresponding element of $\V^*$ is denoted by $\B(\v, \u)$. We also denote $\B(\u) = \B(\u, \u)=\mathcal{P}[(\u\cdot\nabla)\u]$.
	
		\begin{remark}
		Note that $\langle\B(\v,\v-\u),\v-\u\rangle=0$, which  implies that
		\begin{equation}\label{441}
			\begin{aligned}
				\langle \B(\v)-\B(\u),\v-\u\rangle =\langle\B(\v-\u,\u),\v-\u\rangle=-\langle\B(\v-\u,\v-\u),\u\rangle.
			\end{aligned}
		\end{equation}
	\end{remark}	
	
	\begin{remark}\label{TriEsti}
		The following estimate on $b(\cdot,\cdot,\cdot)$ plays a crucial role in the sequel. Applying H\"older's, Ladyzhenskaya's (\cite[Lemma 1, Chapter I]{OAL}), Sobolev's (\cite[Theorem 7, pp. 190]{DMDZ}) and interpolation (\cite[Theorem 6, pp. 200]{DMDZ}) inequalities, and \eqref{EquiH2}, respectively, we obtain
		\begin{align}
			|b(\v,\u,\w)|&\leq\|\v\|_{\L^4(\R^2)}\|\nabla\u\|_{\L^4(\R^2)}\|\w\|_{\H}\nonumber\\&\leq C\|\v\|^{1/2}_{\H}\|\nabla\v\|^{1/2}_{\H}\|\u\|_{\H^{\frac{3}{2}}(\R^2)}\|\w\|_{\H}\nonumber\\&\leq C\|\v\|^{1/2}_{\H}\|\nabla\v\|^{1/2}_{\H}\|\u\|^{1/2}_{\V}\|\u\|^{1/2}_{\H^{2}(\R^2)}\|\w\|_{\H}\nonumber\\&\leq C\|\v\|^{1/2}_{\H}\|\nabla\v\|^{1/2}_{\H}\big(\|\u\|^{1/2}_{\H}+\|\nabla\u\|^{1/2}_{\H}\big)\|(\I+\A)\u\|^{1/2}_{\H}\|\w\|_{\H}\nonumber\\&\leq C\|\v\|^{1/2}_{\H}\|\nabla\v\|^{1/2}_{\H}\|\u\|^{1/2}_{\H}\big(\|\u\|^{1/2}_{\H}+\|\A\u\|^{1/2}_{\H}\big)\|\w\|_{\H}\nonumber\\&\quad+ C\|\v\|^{1/2}_{\H}\|\nabla\v\|^{1/2}_{\H}\|\nabla\u\|^{1/2}_{\H}\big(\|\u\|^{1/2}_{\H}+\|\A\u\|^{1/2}_{\H}\big)\|\w\|_{\H}\label{b1}.
		\end{align}
	\end{remark}

\subsection{Abstract formulation}	Taking the projection $\mathcal{P}$ on SNSE \eqref{1}, we obtain for $t\geq \tau,$ $\tau\in\mathbb{R}$ 
	\begin{equation}\label{SNSE}
		\left\{
		\begin{aligned}
			\d\u&=-\left[\nu \A\u+\B(\u)-\mathcal{P}\f\right]\d t +\sigma\u\d \W, \qquad \text{ in } \R^2\times(\tau,\infty), \\ 
			\u|_{t=\tau}&=\u_{\tau}, \hspace{64mm} x\in \R^2,
		\end{aligned}
		\right.
	\end{equation}
		where  $\sigma>0$ and the stochastic integral is understood in the It\^o sense. Here, $\W(t,\omega)$ is the standard scalar Wiener process on the probability space $(\Omega, \mathscr{F}, \mathbb{P}),$ where $$\Omega=\{\omega\in C(\R;\R):\omega(0)=0\},$$ endowed with the compact-open topology given by the complete metric
		\begin{align*}
			d_{\Omega}(\omega,\omega'):=\sum_{m=1}^{\infty} \frac{1}{2^m}\frac{\|\omega-\omega'\|_{m}}{1+\|\omega-\omega'\|_{m}},\text{ where } \|\omega-\omega'\|_{m}:=\sup_{-m\leq t\leq m} |\omega(t)-\omega'(t)|,
		\end{align*}
		and $\mathscr{F}$ is the Borel sigma-algebra induced by the compact-open topology of $(\Omega,d_{\Omega}),$ $\mathbb{P}$ is the two-sided Wiener measure on $(\Omega,\mathscr{F})$. Also, define $\{\theta_{t}\}_{t\in\R}$ by 
	\begin{equation*}
		\theta_{t}\omega(\cdot)=\omega(\cdot+t) -\omega(t), \ \ \   t\in\R\ \text{ and }\ \omega\in\Omega.
	\end{equation*}
	Hence, $(\Omega,\mathscr{F},\mathbb{P},\{\theta_{t}\}_{t\in\R})$ is a metric DS. We expect a solution to the system \eqref{SNSE} in the following sense:
	\begin{definition}\label{GWS}
		A stochastic process $\u(t):=\u(t;\tau,\omega,\u_{\tau})$ is a $\textsl{global (analytic) weak solution}$ of the system \eqref{1} for $t\geq\tau$ and $\omega\in\Omega$ if for $\mathbb{P}$-a.s. $\omega\in\Omega$
		\begin{align*}
			\u(\cdot)\in\mathrm{C}([\tau,\infty);\H)\cap\mathrm{L}^{\infty}_{\mathrm{loc}}(\tau,\infty;\V)
		\end{align*}
		and 
		\begin{align*}
			&(\u(t),\boldsymbol{\psi})+\nu\int_{\tau}^{t}(\nabla\u(\xi),\nabla\boldsymbol{\psi})\d\xi+\int_{\tau}^{t}b(\u(\xi),\u(\xi),\boldsymbol{\psi})\d\xi\nonumber\\&=(\u_{\tau},\boldsymbol{\psi})+\int_{\tau}^{t}(\boldsymbol{f}(\xi),\boldsymbol{\psi})\d\xi +\sigma\int_{\tau}^{t}(\u(\xi),\boldsymbol{\psi})\d\W(\xi),
		\end{align*}
		for every $t\geq\tau$ and for every $\boldsymbol{\psi}\in\V$.
	\end{definition}
	
	Let us now consider
\begin{align}\label{OU1}
	y(\theta_{t}\omega) =  \int_{-\infty}^{t} e^{-(t-\xi)}\d \W(\xi), \ \ \omega\in \Omega,
\end{align} which is the \textsl{stationary solution} of the one dimensional \textsl{Ornstein-Uhlenbeck equation}
\begin{align}\label{OU2}
	\d y(\theta_t\omega) +  y(\theta_t\omega)\d t =\d\W(t).
\end{align}
It is known from \cite{FAN} that there exists a $\theta$-invariant subset $\widetilde{\Omega}\subset\Omega$ of full measure such that $y(\theta_t\omega)$ is continuous in $t$ for every $\omega\in \widetilde{\Omega},$ and
\begin{align}
\lim_{t\to\infty}\frac{|\omega(t)|}{t}=	\lim_{t\to \pm \infty} \frac{|y(\theta_t\omega)|}{|t|}&=	\lim_{t\to \pm \infty} \frac{1}{t} \int_{0}^{t} y(\theta_{\xi}\omega)\d\xi =\lim_{t\to \infty} e^{-\delta t}|y(\theta_{-t}\omega)| =0,\label{Z3}
\end{align}
for all $\delta>0$. For further analysis of this work, we do not distinguish between $\widetilde{\Omega}$ and $\Omega$.
Since, $\omega(\cdot)$ has sub-exponential growth  (cf. \cite[Lemma 11]{CGSV}), $\Omega$ can be written as $\Omega=\bigcup\limits_{N\in\N}\Omega_{N}$, where
\begin{align*}
	\Omega_{N}:=\{\omega\in\Omega:|\omega(t)|\leq Ne^{|t|},\ \text{ for all }\ t\in\R\}, \ \text{ for every } \ N\in\N.
\end{align*}
Moreover, for each $N\in\N$, $(\Omega_{N},d_{\Omega_{N}})$ is a polish space (cf. \cite[Lemma 17]{CGSV}).
\begin{lemma}\label{conv_z}
	For each $N\in\N$, suppose $\omega_k,\omega_0\in\Omega_{N}$ are such that $d_{\Omega}(\omega_k,\omega_0)\to0$ as $k\to+\infty$. Then, for each $\tau\in\R$ and $T\in\R^+$ ,
	\begin{align}
		\sup_{t\in[\tau,\tau+T]}&\bigg[|y(\theta_{t}\omega_k)-y(\theta_{t}\omega_0)|+|e^{ y(\theta_{t}\omega_k)}-e^{ y(\theta_{t}\omega_0)}|\bigg]\to 0 \ \text{ as } \ k\to+\infty,\nonumber\\
		\sup_{k\in\N}\sup_{t\in[\tau,\tau+T]}&|y(\theta_{t}\omega_k)|\leq C(\tau,T,\omega_0).\label{conv_z2}
	\end{align}
\end{lemma}
\begin{proof}
	See Corollary 22 and Lemma 2.5 in \cite{CLL} and \cite{YR}, respectively.
\end{proof}

	Now, define $z(t,\omega):=e^{-\sigma y(\theta_{t}\omega)}$ and a new function $\v$  by 
	\begin{align}\label{COV}
		\v(t;\tau,\omega,\v_{\tau})=\z(t,\omega)\u(t;\tau,\omega,\u_{\tau}) \ \ \
		\text{ with }
		\ \ \	\v_{\tau}=\z(\tau,\omega)\u_{\tau},
	\end{align}
where $\u(t;\tau,\omega,\u_{\tau})$ and $y(\theta_{t}\omega)$ are the solutions of \eqref{SNSE} (in the sense of Definition \ref{GWS}) and \eqref{OU2}, respectively. Then $\v(\cdot;\tau,\omega,\v_{\tau})$ satisfies the following:
\begin{equation}\label{CNSE}
		\left\{
		\begin{aligned}
			\frac{\d\v}{\d t}+\nu \A\v+\left[\frac{\sigma^2}{2}-\sigma y(\theta_{t}\omega)\right]\v+\z^{-1}(t,\omega)\B\big(\v\big)&=\z(t,\omega)\mathcal{P}\f , \quad t> \tau, \\ 
			\v|_{t=\tau}&=\v_{\tau}, \qquad \quad x\in \R^2,
		\end{aligned}
		\right.
	\end{equation}
	in $\V^*$. An application of the standard Faedo-Galerkin approximation method ensures that for all $t\geq \tau,\ \tau\in\R,$ for every $\v_{\tau}\in\H$ and $\f\in\mathrm{L}^{2}_{\text{loc}}(\R;\H^{-1}(\R^2))$, \eqref{CNSE} has a unique weak solution $\v\in\mathrm{C}([\tau,+\infty);\H)\cap\mathrm{L}^2_{\mathrm{loc}}(\tau,+\infty;\V)$ (cf. \cite{R.Temam}). Hence, for every $\u_{\tau}\in\H$ and $\f\in\mathrm{L}^{2}_{\text{loc}}(\R;\H^{-1}(\R^2))$, there exists a unique solution $\u(\cdot)$ to the system \eqref{SNSE} in the sense of Definition \ref{GWS}. Furthermore, if $\f\in\mathrm{L}^{2}_{\text{loc}}(\R;\L^2(\R^2))$,  then one can show that $\v\in \mathrm{C}((\tau,+\infty);\V)$ (cf. \cite[Corollary 4]{Robinson}, see Lemma \ref{lem3.6} below). For $t\in\R^+, \tau\in\R, \omega\in\Omega$ and $\u_{\tau}\in\H$, define a map $\Phi:\R^+\times\R\times\Omega\times\H\to\H$  by 
	\begin{align}\label{Phi1}
		\Phi(t,\tau,\omega,\u_{\tau}) =\u(t+\tau;\tau,\theta_{-\tau}\omega,\u_{\tau})=\frac{\v(t+\tau;\tau,\theta_{-\tau}\omega,\v_{\tau})}{\z(t+\tau,\theta_{-\tau}\omega)},
	\end{align}
	with $\v_{\tau}=\z(\tau,\theta_{-\tau}\omega)\u_{\tau}$. Note that $\Phi$ is a random cocycle (cf. \cite{PeriodicWang}). Assume that $D=\{D(\tau,\omega):\tau\in\R,\omega\in\Omega\}$ is a family of non-empty subsets of $\H$ satisfying, for every $c>0, \tau\in\R$ and $\omega\in\Omega$, 
	\begin{align}\label{D_1}
		\lim_{t\to\infty}e^{-ct}\|D(\tau-t,\theta_{-t}\omega)\|^2_{\H}=0.
	\end{align}
	Let us define the \emph{universe} of tempered subsets of $\H$ as $$\mathfrak{D}:=\{D=\{D(\tau,\omega):\tau\in\R\text{ and }\omega\in\Omega\}:D \text{ satisfying } \eqref{D_1}\}.$$

\begin{remark}
	One can use an another transformation  to change the system \eqref{SNSE} into an equivalent pathwise deterministic  system.  For a given $t\in\R$ and $\omega\in \Omega$, let $\hat{z}(t,\omega)=e^{-\sigma\omega(t)}$. Then, $\hat{z}$ satisfies the equation 
	\begin{align}\label{Trans1}
		\d \hat{z}=\frac{\sigma^2}{2} \hat{z} \d t-\sigma \hat{z}\d \W.
	\end{align}
	Let $\v$ be a new function given by 
	\begin{align}\label{COV2}
		\v(t;\tau,\omega,\v_{\tau})= \hat{z}(t,\omega)\u(t;\tau,\omega,\u_{\tau}) \ \ \
		\text{ with }
		\ \ \	\v_{\tau}= \hat{z}(\tau,\omega)\u_{\tau},
	\end{align}
	where $\u(t;\tau,\omega,\u_{\tau})$ and $ \hat{z}(t,\omega)$ are the solutions of \eqref{SNSE} and \eqref{Trans1}, respectively. Then $\v(\cdot;\tau,\omega,\v_{\tau})$ satisfies the following system:
	\begin{equation}\label{CNSE1}
		\left\{
		\begin{aligned}
			\frac{\d\v}{\d t}+\nu \A\v+\frac{\sigma^2}{2}\v+\frac{1}{\hat{z}(t,\omega)}\B\big(\v\big)&= \hat{z}(t,\omega)\f , \quad t> \tau, \\ 
			\v|_{t=\tau}&=\v_{\tau}, \qquad \quad x\in \R^2,
		\end{aligned}
		\right.
	\end{equation}
	which is a pathwise deterministic system and is equivalent to the system \eqref{SNSE}.  Even though  the system \eqref{CNSE1} is easier to handle than  the system \eqref{CNSE}, there is a technical difficulty in using the transformation  \eqref{COV2}. Since $\omega(t)$ is identified with $\mathrm{W}(t,\omega)$, the change of variable given by \eqref{COV2} involves the Wiener process explicitly. Therefore, $\hat{z}(t,\omega)$ is not a stationary process and the change of variable given by \eqref{COV2} is not stationary. This fact creates troubles if one uses the change of variable given by \eqref{COV2} in order to have conjugated random dynamical systems. Since $z(t,\omega)=e^{-\sigma y(\theta_{t}\omega)}$ is stationary, the change of variable is given by a homeomorphism (also called conjugation) which transforms one random dynamical system in an another equivalent one (cf. \cite[Proposition B.9]{GLS2}). For this reason, it is appropriate to use the change of variable mentioned in \eqref{COV}.  We also point out here that the authors in \cite{HV} used the transformation given \eqref{Trans1} to study the global (or local) existence of smooth pathwise solutions 2D (or 3D) Euler equations. But they are not investigating  any qualitative properties of the solutions (cf. \cite{MRRZXZ} also). 
\end{remark}

	\section{Bi-spatial Pullback Random Attractor for SNSE}\label{sec3}\setcounter{equation}{0}
	In this section, we establish  the existence of unique $(\H,\V)$-pullback random attractors (belonging to class $\mathfrak{D}$) for non-autonomous SNSE on the whole space $\R^2$. In order to prove the existence of unique bi-spatial pullback random attractors, we use an abstract result established in the work \cite{WZhao}. 
	For the basic definition of the terms that are used in the following theorem, the readers are referred to see \cite[Section 2]{WZhao}.

	\begin{theorem}[Theorem 2.10, \cite{WZhao}]\label{AbstractResult}
		Let $(\mathbb{X},\d_{\mathbb{X}})$ and $(\mathbb{Y},\d_{\mathbb{Y}})$ be two complete metric spaces. Let $\Phi$ be a random cocycle on $\X$ (over metric DS $(\Omega,\mathscr{F},\mathbb{P},\{\theta_{t}\}_{t\in\R})$) which is $(\X,\mathbb{Y})$-continuous and $\mathfrak{D}$ be inclusion closed universe in $\X$. Suppose that 
		\begin{itemize}
			\item[(i)] $\Phi$ has a closed pullback $\mathfrak{D}$-random absorbing set $\mathcal{K}=\{\mathcal{K}(\tau,\omega): \tau\in\R \text{ and }\omega\in \Omega\}\in\mathfrak{D}$ in $\X$;
			\item [(ii)] $\Phi$ is $(\X,\X)$-pullback asymptotically compact in $\X$;
			\item[(iii)] For every fixed $t>0$, $\tau\in\R$ and $\x\in\X$, the mapping $\Phi(t,\tau,\cdot,\x):\Omega\to\mathbb{Y}$ is $(\mathscr{F},\mathscr{B}(\mathbb{Y}))$-measurable.
		\end{itemize}
	Then the random cocycle $\Phi$ possesses a unique $(\X,\mathbb{Y})$-pullback random attractor $$\mathscr{A}=\{\mathscr{A}(\tau,\omega): \tau\in\R \text{ and }\omega\in \Omega\}\in\mathfrak{D},$$ where
	\begin{align*}
	\mathscr{A}(\tau,\omega)=\bigcap_{s>0}\overline{\bigcup_{t\geq s}\Phi(t,\tau-t,\theta_{-t}\omega,\mathcal{K}(\tau-t,\theta_{-t}\omega))}^{\X}.
	\end{align*}
Moreover, it can also be structured by the $\mathbb{Y}$-metric, that is,
	\begin{align*}
	\mathscr{A}(\tau,\omega)=\bigcap_{s>0}\overline{\bigcup_{t\geq s}\Phi(t,\tau-t,\theta_{-t}\omega,\mathcal{K}(\tau-t,\theta_{-t}\omega))}^{\mathbb{Y}}.
\end{align*}
	\end{theorem}
	
	The following assumption on the external forcing term $\f$ is needed to prove the results of this section. 
	
	\begin{hypothesis}\label{Hyp-f}
		For the external forcing term $\f\in\mathrm{L}^{2}_{\emph{loc}}(\R;\L^2(\R^2))$, there exists a number $\delta\in[0,\frac{\sigma^2}{2})$ such that for every $c>0$,
		\begin{align}\label{forcing2}
			\lim_{s\to-\infty}e^{cs}\int_{-\infty}^{0} e^{\delta\zeta}\|\f(\cdot,\zeta+s)\|^2_{\L^2(\R^2)}\d \zeta=0.
		\end{align}
	\end{hypothesis}
	A direct consequence of  Hypothesis \ref{Hyp-f} is as follows:
	\begin{proposition}[Proposition 4.2, \cite{KM6}]\label{Hypo_Conseq}
		Assume that Hypothesis \ref{Hyp-f} holds. Then
		\begin{align}\label{forcing1}
			\int_{-\infty}^{\tau} e^{\delta\zeta}\|\f(\cdot,\zeta)\|^2_{\L^2(\R^2)}\d \zeta<\infty, \ \ \text{ for all } \tau\in\R,
		\end{align}
		where $\delta$ is the same as in \eqref{forcing2}.
	\end{proposition}
	\begin{example}
		Take $\f(\cdot,t)=t^{p}\f_1$, for any $p\geq0$ and $\f_1\in\L^2(\R^2)$. Note that the conditions \eqref{forcing2}-\eqref{forcing1} do not need $\f$ to be bounded in $\L^2(\R^2)$ at $\pm\infty$.
	\end{example}
\begin{remark}\label{rem3.13}
	One can prove the existence of unique pullback random attractors in $\H$ by replacing $\L^2(\R^2)$ with $\H^{-1}(\R^2)$ in Hypothesis \ref{Hyp-f} (cf. \cite{PeriodicWang}). But, in order to prove the existence of pullback random attractors in a more regular space $\V$, one has to consider Hypothesis \ref{Hyp-f}.
\end{remark}
	The following lemma helps us to prove the $(\H,\V)$-continuity of random cocycle $\Phi$ as well as the $(\mathscr{F},\mathscr{B}(\V))$-measurability of $\Phi(t,\tau,\cdot,\x):\Omega\to\mathbb{V}$, for every fixed $t>0$, $\tau\in\R$ and $\x\in\H$.
	\begin{lemma}\label{lem3.6}
			For $\v_{\tau}\in\H$ and $\f\in\mathrm{L}^2_{\mathrm{loc}}(\R;\L^2(\R^2))$, there exist three random variables $\rho(t,\tau,\omega,\v_{\tau},\f)$, $\widetilde{\rho}(t,\tau,\omega,\v_{\tau},\f)$ and $\widehat{\rho}(t,\tau,\omega,\v_{\tau},\f)$ such that 
			\begin{align}
				\|\v(t)\|^2_{\H}+\frac{\sigma^2}{2}\int_{\tau}^{t}\|\v(\zeta)\|^2_{\H}\d\zeta+2\nu\int_{\tau}^{t}\|\nabla\v(\zeta)\|^2_{\H}\d\zeta\leq\rho(t,\tau,\omega,\v_{\tau},\f),\ \text{ for all } \ t\geq\tau,
			\end{align}
		and
			\begin{eqnarray}\label{ei8}
				\begin{aligned}
		\left\{	\begin{array}{l}\|\nabla\v(t)\|_{\H}^2\leq \widetilde{\rho}(t,\tau,\omega,\v_{\tau},\f) \ \text{ and } \ \\	\int\limits_{\frac{\tau+t}{2}}^{t}\|\A\v(\zeta)\|^2_{\H}\d\zeta \leq \widehat{\rho}(t,\tau,\omega,\v_{\tau},\f), \ \text{ for all } \ t>\tau.\end{array}\right.
		\end{aligned}
		\end{eqnarray}
	\end{lemma}
\begin{proof}
		From the first equation of the system \eqref{CNSE}, \eqref{b0} and Young's inequality, we obtain
	\begin{align}\label{ei1}
		\frac{\d}{\d t} \|\v(t)\|^2_{\H}+ \frac{\sigma^2}{2}\|\v(t)\|^2_{\H} +2\nu\|\nabla\v(t)\|^2_{\H} \leq \frac{2\z^2(t,\omega)}{\sigma^2}\|\f(t)\|^2_{\L^2(\R^2)} + 2 \sigma|y(\theta_{t}\omega)|\|\v(t)\|^2_{\H},
	\end{align}
for a.e. $t\geq\tau$, which gives (using Gronwall's inequality)
\begin{align}\label{ei22}
	\|\v(t)\|^2_{\H}&\leq\left[\|\v_{\tau}\|^2_{\H}+\frac{2}{\sigma^2}\int_{\tau}^{t} \z^2(\zeta,\omega) \|\f(\zeta)\|^2_{\L^2(\R^2)}\d \zeta\right] e^{\int_{\tau}^{t}2\sigma| y(\theta_{\zeta}\omega)|\d\zeta}, \ \text{ for all } \ t\geq \tau.
\end{align}
From \eqref{ei1} and \eqref{ei22}, we write
\begin{align}\label{ei2}
	&\|\v(t)\|^2_{\H}+\frac{\sigma^2}{2}\int_{\tau}^{t}\|\v(\zeta)\|^2_{\H}\d\zeta+2\nu\int_{\tau}^{t}\|\nabla\v(\zeta)\|^2_{\H}\d\zeta\nonumber\\&\leq \|\v_{\tau}\|^2_{\H}+\frac{2}{\sigma^2}\int_{\tau}^{t} \z^2(\zeta,\omega) \|\f(\zeta)\|^2_{\L^2(\R^2)}\d \zeta+\int_{\tau}^{t}2\sigma|y(\theta_{\zeta}\omega)|\|\v(\zeta)\|^2_{\H}\d\zeta\nonumber\\&:=\rho(t,\tau,\omega,\v_{\tau},\f),
\end{align}
for all $t\geq\tau$. Taking the inner product of the first equation in \eqref{CNSE} with $\A\v(\cdot)$, using \eqref{b1} and Young's inequality, we find
\begin{align}\label{ei3}
		&\frac{1}{2}\frac{\d}{\d t} \|\nabla\v(t)\|^2_{\H}+ \frac{\sigma^2}{2}\|\nabla\v(t)\|^2_{\H} +\nu\|\A\v(t)\|^2_{\H}\nonumber\\&=\z^{-1}(t,\omega)\ b(\v(t),\v(t),\A\v(t))+\z(t,\omega)(\f(t),\A\v(t))+\sigma y(\theta_{t}\omega)\|\nabla\v(t)\|^2_{\H}\nonumber\\&\leq\frac{\nu}{2}\|\A\v(t)\|^2_{\H}+\frac{\sigma^2}{4}\|\nabla\v(t)\|^2_{\H}+C\z^{-4}(t,\omega)\|\v(t)\|^{2}_{\H}\|\nabla\v(t)\|^4_{\H}+C\z^{-4}(t,\omega)\|\v(t)\|^{6}_{\H}\nonumber\\&\quad+C\z^{-4}(t,\omega)\|\v(t)\|^{4}_{\H}\|\nabla\v(t)\|^2_{\H}+C\z^{-2}(t,\omega)\|\v(t)\|^{2}_{\H}\|\nabla\v(t)\|^2_{\H}+\sigma|y(\theta_{t}\omega)|\|\nabla\v(t)\|^2_{\H}\nonumber\\&\quad+C\z^2(t,\omega)\|\f(t)\|^2_{\L^2(\R^2)},
\end{align}
for a.e. $t\geq\tau$. From \eqref{ei3}, we infer
\begin{align}\label{ei4}
	&\frac{\d}{\d t} \|\nabla\v(t)\|^2_{\H}\nonumber\\&\leq C\z^{-4}(t,\omega)\|\v(t)\|^{2}_{\H}\|\nabla\v(t)\|^4_{\H}+C\z^{-4}(t,\omega)\|\v(t)\|^{6}_{\H}+C\z^{-4}(t,\omega)\|\v(t)\|^{4}_{\H}\|\nabla\v(t)\|^2_{\H}\nonumber\\&\quad+C\z^{-2}(t,\omega)\|\v(t)\|^{2}_{\H}\|\nabla\v(t)\|^2_{\H}+2\sigma|y(\theta_{t}\omega)|\|\nabla\v(t)\|^2_{\H}+C\z^2(t,\omega)\|\f(t)\|^2_{\L^2(\R^2)},
\end{align}
for a.e. $t\geq\tau$. For $t>\tau$, using \eqref{ei2}, we have 
\begin{align}
	&\int_{\tau}^{t}\|\nabla\v(\zeta)\|^2_{\H}\d\zeta\leq\rho(t,\tau,\omega,\v_{\tau},\f),\label{ei5}\\ \nonumber
	&C\int_{\tau}^{t}\z^{-4}(t,\omega)\|\v(\zeta)\|^2_{\H}\|\nabla\v(\zeta)\|^2_{\H}\d\zeta+2\sigma\int_{\tau}^{t}|y(\theta_{\zeta}\omega)|\d\zeta\nonumber\\&\leq C\sup_{\zeta\in[\tau,t]}\left\{\z^{-4}(\zeta,\omega)\|\v(\zeta)\|^2_{\H}\right\}\int_{\tau}^{t}\|\nabla\v(\zeta)\|^2_{\H}\d\zeta+2\sigma\int_{\tau}^{t}|y(\theta_{\zeta}\omega)|\d\zeta\nonumber\\&\leq C\sup_{\zeta\in[\tau,t]}\left\{\z^{-4}(\zeta,\omega)\right\}[\rho(t,\tau,\omega,\v_{\tau},\f)]^2+2\sigma\int_{\tau}^{t}|y(\theta_{\zeta}\omega)|\d\zeta,\nonumber\\&:=\rho_1(t,\tau,\omega,\v_{\tau},\f),\label{ei6}
\end{align}
and
\begin{align}
	&C\int_{\tau}^{t}\z^{-4}(\zeta,\omega)\|\v(\zeta)\|^{6}_{\H}\d\zeta+C\int_{\tau}^{t}\z^{-4}(\zeta,\omega)\|\v(\zeta)\|^{4}_{\H}\|\nabla\v(\zeta)\|^2_{\H}\d\zeta\nonumber\\&\quad+C\int_{\tau}^{t}\z^{-2}(\zeta,\omega)\|\v(\zeta)\|^{2}_{\H}\|\nabla\v(\zeta)\|^2_{\H}\d\zeta+C\int_{\tau}^{t}\z^2(\zeta,\omega)\|\f(\zeta)\|^2_{\L^2(\R^2)}\d\zeta\nonumber\\&\leq C\sup\limits_{\zeta\in[\tau,t]}\z^{-4}(\zeta,\omega)[\rho(t,\tau,\omega,\v_{\tau},\f)]^3+C\sup\limits_{\zeta\in[\tau,t]}\z^{-2}(\zeta,\omega)[\rho(t,\tau,\omega,\v_{\tau},\f)]^2\nonumber\\&\quad+2\sigma\int_{\tau}^{t}|y(\theta_{\zeta}\omega)|\d\zeta+C\sup\limits_{\zeta\in[\tau,t]}\z^2(\zeta,\omega)\int_{\tau}^{t}\|\f(\zeta)\|^2_{\L^2(\R^2)}\d\zeta\nonumber\\&:=\rho_2(t,\tau,\omega,\v_{\tau},\f).\label{ei7}
\end{align}
In view of uniform Gronwall's lemma (cf. \cite[Lemma 1.1, pp. 91]{R.Temam}) along with \eqref{ei5}-\eqref{ei7}, we deduce 
\begin{align*}
	\|\nabla\v(t)\|_{\H}^2&\leq \bigg[\frac{\rho(t,\tau,\omega,\v_{\tau},\f)}{t-\tau}+\rho_2(t,\tau,\omega,\v_{\tau},\f)\bigg]\exp\{\rho_1(t,\tau,\omega,\v_{\tau},\f)\}\nonumber\\&:=\widetilde{\rho}(t,\tau,\omega,\v_{\tau},\f), \ \text{ for all }\  t>\tau.
\end{align*}
Moreover, \eqref{ei3} gives
\begin{align*}
	\int_{\frac{\tau+t}{2}}^{t}\|\A\v(\zeta)\|^2_{\H}\d\zeta& \leq \widetilde{\rho}\left(\frac{\tau+t}{2},\tau,\omega,\v_{\tau},\f\right)+C\{\rho(t,\tau,\omega,\v_{\tau},\f)\}^2\sup_{\zeta\in[\frac{\tau+t}{2},t]}\bigg[\frac{\widetilde{\rho}(\zeta,\tau,\omega,\v_{\tau},\f)}{\z^4(\zeta,\omega)}\bigg]\nonumber\\&\quad+2\sigma\sup_{\zeta\in[\frac{\tau+t}{2},t]}\left[\widetilde{\rho}(\zeta,\tau,\omega,\v_{\tau},\f)\right]\int_{\frac{\tau+t}{2}}^{t}|y(\theta_{\zeta}\omega)|\d\zeta+\rho_2(t,\tau,\omega,\v_{\tau},\f)\\&:= \widehat{\rho}(t,\tau,\omega,\v_{\tau},\f), \ \text{ for all }\  t>\tau,
\end{align*} 
which completes the proof.
\end{proof}

In order to apply the abstract result stated in Theorem \ref{AbstractResult}, $\Phi$ should be $(\H,\V)$-continuous. The following lemma shows that our random cocycle is $(\H,\V)$-continuous.
	\begin{lemma}\label{Continuity}
	Assume that $\f\in \mathrm{L}^2_{\emph{loc}}(\mathbb{R};\L^2(\R^2))$. Then, the solution of \eqref{CNSE} is continuous in $\V$ with respect to the  initial data in $\H.$
\end{lemma}
\begin{proof}
	Let $\v_1(t):=\v_{1}(t;\tau,\omega,\v_{1,\tau})$ and $\v_{2}(t):=\v_{2}(t;\tau,\omega,\v_{2,\tau})$ be two solutions of the system \eqref{CNSE}. Then $\mathscr{U}(\cdot)=\v_{1}(\cdot)-\v_{2}(\cdot)$ with $\mathscr{U}(\tau)=\v_{1,\tau}-\v_{2,\tau}$ satisfies
	\begin{align}\label{Conti1}
		\frac{\d\mathscr{U}(t)}{\d t}+\nu \A\mathscr{U}(t)+\left[\frac{\sigma^2}{2}-\sigma y(\theta_{t}\omega)\right]	\mathscr{U}(t)=-\z^{-1}(t,\omega)\left\{\B\big(\v_1(t)\big)-\B\big(\v_2(t)\big)\right\},
	\end{align}
	for a.e. $t\geq\tau$ in $\V^*$. Multiplying \eqref{Conti1} with $\mathscr{U}(\cdot)$ and then integrating over $\R^2$, we infer
	\begin{align}\label{Conti2}
		&\frac{1}{2}\frac{\d}{\d t} \|\mathscr{U}(t)\|^2_{\H} +\nu \|\nabla\mathscr{U}(t)\|^2_{\H} + \frac{\sigma^2}{2}\|\mathscr{U}(t)\|^2_{\H} \nonumber\\&=-\z^{-1}(t,\omega)\left\langle\B\big(\v_1(t)\big)-\B\big(\v_2(t)\big), \mathscr{U}(t)\right\rangle+\sigma y(\theta_{t}\omega)\|\mathscr{U}(t)\|^2_{\H},
	\end{align}
	for a.e. $t\geq\tau$. Using \eqref{441}, H\"older's, Ladyzhenskaya's and Young's inequalities, we obtain
	\begin{align}\label{Conti3}
		\left| \z^{-1}(t,\omega)\left\langle\B\big(\v_1\big)-\B\big(\v_2\big), \mathscr{U}\right\rangle\right|&=\left|\z^{-1}(t,\omega)\left\langle\B\big(\mathscr{U},\mathscr{U} \big), \v_2\right\rangle\right|\nonumber\\&\leq\frac{\nu}{2}\|\nabla\mathscr{U}\|^2_{\H}+C \z^{-4}(t,\omega)\|\v_2\|^4_{\L^4}\|\mathscr{U}\|^2_{\H}\nonumber\\&\leq\frac{\nu}{2}\|\nabla\mathscr{U}\|^2_{\H}+C\z^{-4}(t,\omega)\|\v_2\|^2_{\H}\|\nabla\v_2\|^2_{\H}\|\mathscr{U}\|^2_{\H}.
	\end{align}
	Making use of \eqref{Conti3} in \eqref{Conti2}, we arrive at
	\begin{align}\label{Conti4}
		&	\frac{\d}{\d t} \|\mathscr{U}(t)\|^2_{\H} +\nu\|\nabla\mathscr{U}(t)\|^2_{\H} \leq \left[C\z^{-4}(t,\omega)\|\v_2(t)\|^2_{\H}\|\nabla\v_2(t)\|^2_{\H}+\sigma |y(\theta_{t}\omega)|\right]\|\mathscr{U}(t)\|^2_{\H},
	\end{align} for a.e. $t\geq\tau$. An application of Gronwall's inequality implies 
	\begin{align}\label{Conti5}
		\|\mathscr{U}(t)\|^2_{\H}&\leq \exp\left\{C\int_{\tau}^{t}\z^{-4}(\zeta,\omega)\|\v_2(\zeta)\|^2_{\H}\|\nabla\v_2(\zeta)\|^2_{\H}\d\zeta+\sigma \int_{\tau}^{t}|y(\theta_{\zeta}\omega)|\d\zeta\right\}\|\mathscr{U}(\tau)\|^2_{\H}\nonumber\\&\leq\exp\left\{C\sup_{\zeta\in[\tau,t]}\left[\z^{-4}(\zeta,\omega)\|\v_2(\zeta)\|^2_{\H}\right]\int_{\tau}^{t}\|\nabla\v_2(\zeta)\|^2_{\H}\d\zeta+\sigma \int_{\tau}^{t}|y(\theta_{\zeta}\omega)|\d\zeta\right\}\|\mathscr{U}(\tau)\|^2_{\H}\nonumber\\&\leq\exp\Bigg\{C\sup\limits_{\zeta\in[\tau,t]}\z^{-4}(\zeta,\omega)[\rho(t,\tau,\omega,\v_{2,\tau},\f)]^2+\sigma \int_{\tau}^{t}|y(\theta_{\zeta}\omega)|\d\zeta\Bigg\}\|\mathscr{U}(\tau)\|^2_{\H}\nonumber\\&:=\rho_5(t,\tau,\omega,\v_{2,\tau},\f)\|\mathscr{U}(\tau)\|^2_{\H},
	\end{align} 
	for all $t\geq\tau$. Furthermore, integrating \eqref{Conti4} over $[\tau,t]$ with $t>\tau$, we infer
	\begin{align}\label{Conti6}
		&\int_{\tau}^{t}\|\nabla\mathscr{U}(\zeta)\|^2_{\H}\d\zeta\nonumber\\ &\leq \|\mathscr{U}(\tau)\|^2_{\H} +C\int_{\tau}^{t} \left[\z^{-4}(\zeta,\omega)\|\v_2(\zeta)\|^2_{\H}\|\nabla\v_2(\zeta)\|^2_{\H}+|y(\theta_{\zeta}\omega)|\right]\|\mathscr{U}(\zeta)\|^2_{\H}\d\zeta\nonumber\\&\leq \Bigg[1+C\left\{\sup\limits_{\zeta\in[\tau,t]}\z^{-4}(\zeta,\omega)[\rho(t,\tau,\omega,\v_{2,\tau},\f)]^2+\int_{\tau}^{t}|y(\theta_{\zeta}\omega)\d\zeta\right\}\rho_5(t,\tau,\omega,\v_{2,\tau},\f)\Bigg]\|\mathscr{U}(\tau)\|^2_{\H}\nonumber\\&:=\rho_6(t,\tau,\omega,\v_{2,\tau},\f)\|\mathscr{U}(\tau)\|^2_{\H},
	\end{align}
	for all $t>\tau$. 
	
	Taking the inner product of \eqref{Conti1} with $\A\mathscr{U}(\cdot),$ and using \eqref{b1} and Young's inequality, we find
	\begin{align}\label{Conti7}
		&\frac{1}{2}\frac{\d}{\d t} \|\nabla\mathscr{U}(t)\|^2_{\H} +\nu \|\A\mathscr{U}(t)\|^2_{\H} + \frac{\sigma^2}{2}\|\nabla\mathscr{U}(t)\|^2_{\H} \nonumber\\&=-\z^{-1}(t,\omega)\left\langle\B\big(\v_1(t)\big)-\B\big(\v_2(t)\big),\A \mathscr{U}(t)\right\rangle+\sigma y(\theta_{t}\omega)\|\nabla\mathscr{U}(t)\|^2_{\H}\nonumber\\&=-\z^{-1}(t,\omega)\big\{b\big(\mathscr{U}(t),\v_1(t),\A \mathscr{U}(t)\big)+b\big(\v_2(t),\mathscr{U}(t),\A \mathscr{U}(t)\big)\big\}+\sigma y(\theta_{t}\omega)\|\nabla\mathscr{U}(t)\|^2_{\H} \nonumber\\&\leq \frac{\nu}{2} \|\A\mathscr{U}(t)\|^2_{\H} + C\big\{\|\v_1(t)\|^{2}_{\H}+\|\A\v_1(t)\|^{2}_{\H}+\|\nabla\v_1(t)\|^{2}_{\H}+\z^{-2}(t,\omega)\|\v_2(t)\|^2_{\H}\nonumber\\&\quad+\z^{-2}(t,\omega)\|\nabla\v_2(t)\|^2_{\H}+\z^{-4}(t,\omega)\|\v_2(t)\|^2_{\H}\|\nabla\v_2(t)\|^2_{\H}\big\}\|\mathscr{U}(t)\|^2_{\H}+C\z^{-4}(t,\omega)\big\{\|\v_1(t)\|^{2}_{\H}\nonumber\\&\quad+\|\nabla\v_1(t)\|^{2}_{\H}+\|\v_2(t)\|^2_{\H}\|\nabla\v_2(t)\|^2_{\H}+\sigma |y(\theta_{t}\omega)|\big\}\|\nabla\mathscr{U}(t)\|^{2}_{\H}+\sigma |y(\theta_{t}\omega)|\|\nabla\mathscr{U}(t)\|^{2}_{\H},
	\end{align}
	for a.e. $t\geq\tau$.
	Replacing $t$ by $\zeta$ in \eqref{Conti7} and multiplying by $(\zeta-\frac{\tau+t}{2})$ with $\zeta\in[\frac{\tau+t}{2},t]$, we obtain 
	
	\begin{align}\label{Conti8}
		&\bigg[\zeta-\frac{\tau+t}{2}\bigg]\frac{\d}{\d\zeta}\|\nabla\mathscr{U}(\zeta)\|^2_{\H}\nonumber\\& \leq \frac{C}{\z^4(\zeta,\omega)}\bigg[\zeta-\frac{\tau+t}{2}\bigg]\bigg[\|\v_1(\zeta)\|^{2}_{\H}+\|\nabla\v_1(\zeta)\|^{2}_{\H}+\|\v_2(\zeta)\|^2_{\H}\|\nabla\v_2(\zeta)\|^2_{\H}\bigg]\|\nabla\mathscr{U}(\zeta)\|^{2}_{\H}\nonumber\\&\quad+C\bigg[\zeta-\frac{\tau+t}{2}\bigg]\bigg[\|\v_1(\zeta)\|^{2}_{\H}+\|\A\v_1(\zeta)\|^{2}_{\H}+\|\nabla\v_1(\zeta)\|^{2}_{\H}+\frac{\|\v_2(\zeta)\|^2_{\H}}{\z^2(\zeta,\omega)}+\frac{\|\nabla\v_2(\zeta)\|^2_{\H}}{\z^2(\zeta,\omega)}\nonumber\\&\quad+\frac{\|\v_2(\zeta)\|^2_{\H}\|\nabla\v_2(\zeta)\|^2_{\H}}{\z^4(\zeta,\omega)}\bigg]\|\mathscr{U}(\zeta)\|^2_{\H}+ 2\sigma\bigg[\zeta-\frac{\tau+t}{2}\bigg] |y(\theta_{\zeta}\omega)|\|\nabla\mathscr{U}(\zeta)\|^{2}_{\H}.
	\end{align}
	An integration by parts leads  to
	\begin{align}\label{Conti9}
		\int_{\frac{\tau+t}{2}}^{t}\bigg[\zeta-\frac{\tau+t}{2}\bigg]\frac{\d}{\d\zeta}\|\nabla\mathscr{U}(\zeta)\|^2_{\H}\d\zeta= \frac{t-\tau}{2}\|\nabla\mathscr{U}(t)\|^2_{\H}-\int_{\frac{\tau+t}{2}}^{t}\|\nabla\mathscr{U}(\zeta)\|^2_{\H}\d\zeta.
	\end{align}
	Integrating \eqref{Conti8} over $[\frac{\tau+t}{2},t]$ and using \eqref{Conti9}, we arrive at
	\begin{align}\label{Conti10}
		&\frac{t-\tau}{2}\|\nabla\mathscr{U}(t)\|^2_{\H}\nonumber\\& \leq \int_{\frac{\tau+t}{2}}^{t}\|\nabla\mathscr{U}(\zeta)\|^2_{\H}\d\zeta+\frac{C}{2}(t-\tau)\sup_{\zeta\in[\frac{\tau+t}{2},t]}\bigg[\frac{1}{\z^4(\zeta,\omega)}\bigg\{\|\v_1(\zeta)\|^{2}_{\H}+\|\nabla\v_1(\zeta)\|^{2}_{\H}\nonumber\\&\quad+\|\v_2(\zeta)\|^2_{\H}\|\nabla\v_2(\zeta)\|^2_{\H}\bigg\}\bigg]\int_{\frac{\tau+t}{2}}^{t}\|\nabla\mathscr{U}(\zeta)\|^{2}_{\H}\d\zeta+\frac{C}{2}(t-\tau)\sup_{\zeta\in[\frac{\tau+t}{2},t]}\left[\|\mathscr{U}(\zeta)\|^2_{\H}\right]\nonumber\\&\quad\times\int_{\frac{\tau+t}{2}}^{t}\bigg[\|\v_1(\zeta)\|^{2}_{\H}+\|\A\v_1(\zeta)\|^{2}_{\H}+\|\nabla\v_1(\zeta)\|^{2}_{\H}+\frac{\|\v_2(\zeta)\|^2_{\H}}{\z^2(\zeta,\omega)}+\frac{\|\nabla\v_2(\zeta)\|^2_{\H}}{\z^2(\zeta,\omega)}\nonumber\\&\quad+\frac{\|\v_2(\zeta)\|^2_{\H}\|\nabla\v_2(\zeta)\|^2_{\H}}{\z^4(\zeta,\omega)}\bigg]\d\zeta+\sigma(t-\tau) \int_{\frac{\tau+t}{2}}^{t}|y(\theta_{\zeta}\omega)|\|\nabla\mathscr{U}(\zeta)\|^{2}_{\H}\d\zeta\nonumber\\&\leq\rho_6(t,\tau,\omega,\v_{2,\tau},\f)\|\mathscr{U}(\tau)\|^2_{\H}+\frac{C}{2}(t-\tau) \sup_{\zeta\in[\frac{\tau+t}{2},t]}\bigg[\frac{1}{\z^4(\zeta,\omega)}\bigg\{\rho(\zeta,\tau,\omega,\v_{1,\tau},\f)\nonumber\\&\quad+\widetilde{\rho}(\zeta,\tau,\omega,\v_{1,\tau},\f)+\rho(t,\tau,\omega,\v_{2,\tau},\f)\widetilde{\rho}(\zeta,\tau,\omega,\v_{2,\tau},\f)\bigg\}\bigg]\rho_6(t,\tau,\omega,\v_{2,\tau},\f)\|\mathscr{U}(\tau)\|^2_{\H}\nonumber\\&\quad+\frac{C}{2}(t-\tau)\bigg[\widehat{\rho}(t,\tau,\omega,\v_{1,\tau},\f)+\rho(t,\tau,\omega,\v_{1,\tau},\f)+\frac{\rho(t,\tau,\omega,\v_{2,\tau},\f)}{\sup_{\zeta\in[\tau,t]}\{\z(\zeta,\omega)\}^2}\nonumber\\&\quad+\frac{\{\rho(t,\tau,\omega,\v_{2,\tau},\f)\}^2}{\sup_{\zeta\in[\tau,t]}\{\z(\zeta,\omega)\}^4}\bigg]\rho_5(t,\tau,\omega,\v_{2,\tau},\f)\|\mathscr{U}(\tau)\|^2_{\H}+\frac{C}{2}(t-\tau)^2\bigg[\rho(t,\tau,\omega,\v_{1,\tau},\f)\nonumber\\&\quad+\frac{\rho(t,\tau,\omega,\v_{2,\tau},\f)}{\sup_{\zeta\in[\tau,t]}\{\z(\zeta,\omega)\}^2}\bigg]\rho_5(t,\tau,\omega,\v_{2,\tau},\f)\|\mathscr{U}(\tau)\|^2_{\H}\nonumber\\&\quad+\sigma(t-\tau)\sup_{\zeta\in[\tau,t]}|y(\theta_{\zeta}\omega)| \rho_6(t,\tau,\omega,\v_{2,\tau},\f)\|\mathscr{U}(\tau)\|^2_{\H},
	\end{align}
	for all $t>\tau$, where we have used \eqref{ei2}, \eqref{ei8} and \eqref{Conti6}-\eqref{Conti7}. It implies from \eqref{Conti10} that there exist a positive random variable $\rho_7(t,\tau,\omega,\v_{2,\tau},\v_{2,\tau},\f)$ such that for all $t>\tau$,
	\begin{align}
		\|\nabla\mathscr{U}(t)\|^2_{\H} \leq \rho_7(t,\tau,\omega,\v_{1,\tau},\v_{2,\tau},\f)\|\mathscr{U}(\tau)\|^2_{\H},
	\end{align}
	which completes the proof.
\end{proof}

	The following lemma helps us to prove the existence of a pullback $\mathfrak{D}$-random absorbing set. 
	
		\begin{lemma}\label{LemmaUe}
		Assume that Hypothesis \ref{Hyp-f} holds. Then, for every $(\tau,\omega,D)\in\R\times\Omega\times\mathfrak{D}$, there exists $\mathfrak{T}=\mathfrak{T}(\tau,\omega,D)>0$ such that for all $t\geq \mathfrak{T}$ and $s\geq \tau-t$, the solution $\v(\cdot)$ of the system \eqref{CNSE} satisfies (with $\omega$ replaced by $\theta_{-\tau}\omega$)
		\begin{align}\label{ue}
			&\|\v(\tau;\tau-t,\theta_{-\tau}\omega,\v_{\tau-t})\|^2_{\H} + 2\nu\int_{\tau-t}^{\tau}e^{-\int^{\tau}_{\zeta}\left(\frac{\sigma^2}{2}-2\sigma y(\theta_{\upeta-\tau}\omega)\right)\d \upeta}\|\nabla\v(\zeta;\tau-t,\theta_{-\tau}\omega,\v_{\tau-t})\|^2_{\H}\d\zeta\nonumber\\&\leq \frac{4}{\sigma^2}\int_{-\infty}^{0} e^{\frac{\sigma^2}{2}\zeta+2\sigma\int^{0}_{\zeta} y(\theta_{\upeta}\omega)\d \upeta} \z^2(\zeta,\omega) \|\f(\cdot,\zeta+\tau)\|^2_{\L^2(\R^2)}\d \zeta,
		\end{align}
		where $\v_{\tau-t}\in D(\tau-t,\theta_{-t}\omega).$ 
	\end{lemma}
	\begin{proof}
		From the first equation of the system \eqref{CNSE}, \eqref{b0} and Young's inequality, we find 
		\begin{align}\label{ue1}
			\frac{\d}{\d t} \|\v(t)\|^2_{\H}+ \left[\frac{\sigma^2}{2}-2 \sigma y(\theta_{t}\omega)\right]\|\v(t)\|^2_{\H} +2\nu\|\nabla\v(t)\|^2_{\H} \leq \frac{2\z^2(t,\omega)}{\sigma^2}\|\f(t)\|^2_{\L^2(\R^2)},
		\end{align}
		for a.e. $t\geq\tau$. Applying the variation of constant formula to \eqref{ue1} and replacing $\omega$ by $\theta_{-\tau}\omega$ in the above inequality, we obtain
		\begin{align}\label{ue2}
			&\|\v(\tau;\tau-t,\theta_{-\tau}\omega,\v_{\tau-t})\|^2_{\H} + 2\nu\int\limits_{\tau-t}^{\tau}e^{-\int\limits^{\tau}_{\zeta}\left(\frac{\sigma^2}{2}-2\sigma y(\theta_{\upeta-\tau}\omega)\right)\d \upeta}\|\nabla\v(\zeta;\tau-t,\theta_{-\tau}\omega,\v_{\tau-t})\|^2_{\H}\d\zeta\nonumber\\&\leq e^{-\int\limits^{\tau}_{\tau-t}\left(\frac{\sigma^2}{2}-2\sigma y(\theta_{\upeta-\tau}\omega)\right)\d \upeta}\|\v_{\tau-t}\|^2_{\H}+ \frac{2}{\sigma^2}\int\limits_{\tau-t}^{\tau} e^{-\int\limits^{\tau}_{\zeta}\left(\frac{\sigma^2}{2}-2\sigma y(\theta_{\upeta-\tau}\omega)\right)\d \upeta} \z^2(\zeta,\theta_{-\tau}\omega) \|\f(\cdot,\zeta)\|^2_{\L^2(\R^2)}\d \zeta\nonumber\\&\leq e^{-\frac{\sigma^2}{2}t+2\sigma\int\limits^{0}_{-t} y(\theta_{\upeta}\omega)\d \upeta}\|\v_{\tau-t}\|^2_{\H}+ \frac{2}{\sigma^2}\int\limits_{-\infty}^{0} e^{\frac{\sigma^2}{2}\zeta+2\sigma\int\limits^{0}_{\zeta} y(\theta_{\upeta}\omega)\d \upeta} \z^2(\zeta,\omega) \|\f(\cdot,\zeta+\tau)\|^2_{\L^2(\R^2)}\d \zeta.
		\end{align}
		Since $\v_{\tau-t}\in D(\tau-t,\theta_{-t}\omega)$, we have from \eqref{Z3} that for sufficiently large $t>0$
		\begin{align*}
			e^{-\frac{\sigma^2}{2}t+2\sigma\int^{0}_{-t} y(\theta_{\upeta}\omega)\d \upeta}\|\v_{\tau-t}\|^2_{\H}\leq e^{-\frac{\sigma^2}{4} t}\|D(\tau-t,\theta_{-t}\omega)\|^2_{\H}\to 0\ \text{ as } \ t \to \infty.
		\end{align*} 
		Therefore, there exists $\mathfrak{T}=\mathfrak{T}(\tau,\omega,D)>0$ such that 
		\begin{align}\label{ue3}
			e^{-\frac{\sigma^2}{2}t+2\sigma\int^{0}_{-t} y(\theta_{\upeta}\omega)\d \upeta}\|\v_{\tau-t}\|^2_{\H}\leq \frac{2}{\sigma^2}\int_{-\infty}^{0} e^{\frac{\sigma^2}{2}\zeta+2\sigma\int\limits^{0}_{\zeta} y(\theta_{\upeta}\omega)\d \upeta} \z^2(\zeta,\omega) \|\f(\cdot,\zeta+\tau)\|^2_{\L^2(\R^2)}\d \zeta,
		\end{align}
		 for all  $t\geq \mathfrak{T}$, which gives \eqref{ue}.
		Now, it is only left to estimate the final term of \eqref{ue2}. By \eqref{Z3}, we have that there exist $R_1, R_2<0$ such that for all $\zeta\leq R_1$,
		\begin{align*}
			-2\sigma y(\theta_{\zeta}\omega)\leq-\left(\frac{\sigma^2}{4}-\frac{\delta}{2}\right)\zeta,
		\end{align*}
	and for all $\zeta\leq R_2$,
	\begin{align*}
		\frac{\sigma^2}{2}\zeta+2\sigma\int^{0}_{\zeta} y(\theta_{\upeta}\omega)\d \upeta\leq \left(\frac{\sigma^2}{4}+\frac{\delta}{2}\right)\zeta,
	\end{align*}
		where $\delta$ is the positive constant appearing in \eqref{forcing1}. Therefore, 
		\begin{align*}
			[\z(\zeta,{\omega})]^2=e^{-2\sigma y(\theta_{\zeta}\omega)}\leq e^{-\left(\frac{\sigma^2}{4}-\frac{\delta}{2}\right)\zeta},
		\end{align*}
		and we have for all $\zeta\leq R=:\min\{R_1,R_2\}$,
		\begin{align*}
			&e^{\frac{\sigma^2}{2}\zeta+2\sigma\int^{0}_{\zeta} y(\theta_{\upeta}\omega)\d \upeta} \z^2(\zeta,\omega) \|\f(\cdot,\zeta+\tau)\|^2_{\L^2(\R^2)}\leq e^{\delta\zeta} \|\f(\cdot,\zeta+\tau)\|^2_{\L^2(\R^2)}.
		\end{align*}
		Therefore, it follows from \eqref{forcing1} that for every $\tau\in\R$ and $\omega\in\Omega$,
		\begin{align}\label{ue4}
			\int_{-\infty}^{0} e^{\frac{\sigma^2}{2}\zeta+2\sigma\int^{0}_{\zeta} y(\theta_{\upeta}\omega)\d \upeta} \z^2(\zeta,\omega) \|\f(\cdot,\zeta+\tau)\|^2_{\L^2(\R^2)}\d \zeta<\infty.
		\end{align} 
		Hence, from \eqref{ue2}-\eqref{ue4}, \eqref{ue} follows.
	\end{proof}
Next lemma shows the existence of a pullback $\mathfrak{D}$-random absorbing set. 
	
\begin{lemma}\label{PA1}
	Assume that Hypothesis \ref{Hyp-f} holds. Then the continuous cocycle $\Phi$ associated with the system \eqref{SNSE} possesses a closed measurable pullback $\mathfrak{D}$-random absorbing set $\mathcal{K}=\{\mathcal{K}(\tau,\omega):\tau\in\R,\omega\in\Omega\}\in\mathfrak{D}$, where  $\mathcal{K}(\tau,\omega)$ is defined by 
	\begin{align}\label{AB1}
		\mathcal{K}(\tau,\omega)=\{\u\in\H:\|\u\|^2_{\H}\leq\mathcal{M}(\tau,\omega)\},
	\end{align}
	and $\mathcal{M}(\tau,\omega)$ is given by
	\begin{align*}
		\mathcal{M}(\tau,\omega)=\frac{4e^{2y(\omega)}}{\sigma^2}\int_{-\infty}^{0} e^{\frac{\sigma^2}{2}\zeta+2\sigma\int^{0}_{\zeta} y(\theta_{\upeta}\omega)\d \upeta} \z^2(\zeta,\omega) \|\f(\cdot,\zeta+\tau)\|^2_{\L^2(\R^2)}\d \zeta.
	\end{align*}
\end{lemma}
\begin{proof}
	We infer from \eqref{ue4} that
	\begin{align}\label{IRAS2-N}
		\mathcal{M}(\tau,\omega)&= \frac{4e^{2y(\omega)}}{\sigma^2}\int_{-\infty}^{0} e^{\frac{\sigma^2}{2}\zeta+2\sigma\int^{0}_{\zeta} y(\theta_{\upeta}\omega)\d \upeta} \z^2(\zeta,\omega) \|\f(\cdot,\zeta+\tau)\|^2_{\L^2(\R^2)}\d \zeta<\infty.
	\end{align}
	Hence, the absorption follows from Lemma \ref{LemmaUe}. For $c>0$, let $c_1=\min\{\frac{c}{2},\frac{\sigma^2}{4},\delta\}$ and consider
	\begin{align}\label{IRAS3-N}
		&\lim_{t\to+\infty}e^{-ct}\|\mathcal{K}(\tau-t,\theta_{-t}\omega)\|^2_{\H}\nonumber\\&\leq \lim_{t\to+\infty}e^{-ct}\left[\frac{4e^{2y(\theta_{-t}\omega)}}{\sigma^2}\int_{-\infty}^{0} e^{\frac{\sigma^2}{2}\zeta+2\sigma\int^{0}_{\zeta} y(\theta_{\upeta-t}\omega)\d \upeta} \z^2(\zeta-t,\omega) \|\f(\cdot,\zeta+\tau-t)\|^2_{\L^2(\R^2)}\d \zeta\right]\nonumber\\&=\lim_{t\to+\infty}e^{-ct}\left[\frac{4}{\sigma^2}e^{2y(\theta_{-t}\omega)+\int_{0}^{-t}y(\theta_{\upeta}\omega)\d\upeta}\int_{-\infty}^{-t} e^{\frac{\sigma^2}{2}(\zeta+t)-2\sigma\int^{0}_{\zeta} y(\theta_{\upeta-t}\omega)\d \upeta} \z^2(\zeta,\omega) \|\f(\cdot,\zeta+\tau)\|^2_{\L^2(\R^2)}\d \zeta\right] \nonumber\\&\leq\lim_{t\to+\infty}e^{-\left(c-\frac{c_1}{2}\right)t}\left[\frac{4}{\sigma^2}\int_{-\infty}^{-t} e^{\delta \zeta} \|\f(\cdot,\zeta+\tau)\|^2_{\L^2(\R^2)}\d \zeta\right] =0,
	\end{align}
	where we have used \eqref{Z3} and \eqref{forcing1}. It follows from \eqref{IRAS3-N} that $\mathcal{K}\in{\mathfrak{D}}$.
\end{proof}
Next lemma demonstrates the pullback $\mathfrak{D}$-asymptotic compactness of non-autonomous random DS $\Phi$. The proof of the following lemma is analogous to the proof of \cite[Lemma 5.3]{PeriodicWang} with some minor modifications, hence we are omitting it here.
\begin{lemma}\label{PA2}
	Assume that Hypothesis \ref{Hyp-f} holds. Then the continuous cocycle $\Phi$ associated with the system \eqref{SNSE} is $(\H,\H)$-pullback $\mathfrak{D}$-asymptotically compact.
\end{lemma}

The following lemma plays a crucial role in  proving the $(\mathscr{F},\mathscr{B}(\V))$-measurability  of the map $\Phi(t,\tau,\cdot,\x):\Omega\to\mathbb{V}$, for every fixed $t>0$, $\tau\in\R$ and $\x\in\H$.
	 \begin{proposition}\label{LusinC}
	 	Suppose that $\tau\in\R$, $t>\tau$, $\f\in\mathrm{L}^2_{\mathrm{loc}}(\R;\L^2(\R^2))$ and $\v_{\tau}\in\H$. For each $N\in\N$, the mapping $\omega\mapsto\v(t;\tau,\omega,\v_{\tau})$ $($solution of \eqref{CNSE}$)$ is continuous from $(\Omega_{N},d_{\Omega_N})$ to $\V$.
	 \end{proposition}
	 \begin{proof}
	 	Assume that $\omega_k,\omega_0\in\Omega_N,\ N\in\mathbb{N}$ such that $d_{\Omega_N}(\omega_k,\omega_0)\to0$ as $k\to+\infty$. Let $\mathscr{U}^k:=\v^k-\v^0,$ where $\v^k=\v(t;\tau,\omega_k,\v_{\tau})$ and $\v^0=\v(t;\tau,\omega_0,\v_{\tau})$ for $t\geq\tau$. Then, $\mathscr{U}^k$ satisfies:
 	\begin{align}\label{LC1}
	 		\frac{\d\mathscr{U}^k}{\d t}&=-\nu \A\mathscr{U}^k-\left[\frac{\sigma^2}{2}-\sigma y(\theta_{t}\omega_k)\right]\mathscr{U}^k-\z^{-1}(t,\omega_k)\B\big(\v^k\big)+\z^{-1}(t,\omega_0)\B\big(\v^0\big)\nonumber\\&\quad+\left[y(\theta_{t}\omega_k)-y(\theta_{t}\omega_0)\right]\v^0+ \left[\z(t,\omega_k)-\z(t,\omega_0)\right]\mathcal{P}\f,
 	\end{align}
	 	in $\V^*$. Taking the inner product with $\mathscr{U}^k(\cdot)$ in \eqref{LC1}, and using \eqref{b0} and \eqref{441}, we obtain
	 	\begin{align}\label{LC2}
	 		\frac{1}{2}\frac{\d}{\d t}\|\mathscr{U}^k\|^2_{\H}&=-\nu\|\nabla\mathscr{U}^k\|^2_{\H}-\left[\frac{\sigma^2}{2}-\sigma y(\theta_{t}\omega_k)\right]\|\mathscr{U}^k\|^2_{\H}-\z^{-1}(t,\omega_k)b(\mathscr{U}^k,\v^0,\mathscr{U}^k)\nonumber\\&\quad+\left[\z^{-1}(t,\omega_k)-\z^{-1}(t,\omega_0)\right]b(\v^0,\mathscr{U}^k,\v^0)+\left[y(\theta_{t}\omega_k)-y(\theta_{t}\omega_0)\right](\v^0,\mathscr{U}^k)\nonumber\\&\quad+\left[\z(t,\omega_k)-\z(t,\omega_0)\right](\f,\mathscr{U}^k).
	 	\end{align}
	 	Using H\"older's and Young's inequalities, we obtain
	 	\begin{align}
	 		\left|\left[\z(t,\omega_k)-\z(t,\omega_0)\right](\f,\mathscr{U}^k)\right|&\leq C\left|\z(t,\omega_k)-\z(t,\omega_0)\right|^2\|\f\|^2_{\L^2(\R^2)}+\frac{\sigma^2}{8}\|\mathscr{U}^k\|^2_{\H},\label{LC4}\\
	 		\left|\left[y(\theta_{t}\omega_k)-y(\theta_{t}\omega_0)\right](\v^0,\mathscr{U}^k)\right|&\leq C\left|y(\theta_{t}\omega_k)-(\theta_{t}\omega_0)\right|^2\|\v^0\|^2_{\H}+\frac{\sigma^2}{8}\|\mathscr{U}^k\|^2_{\H}.\label{LC44}
	 	\end{align}
	 	Applying H\"older's, Ladyzhenskaya's and Young's inequalities, we estimate
	 	\begin{align}
	 		\left|\z^{-1}(t,\omega_k)b(\mathscr{U}^k,\v^0,\mathscr{U}^k)\right|&\leq C\z^{-2}(t,\omega_k)\|\nabla\v^0\|^2_{\H}\|\mathscr{U}^k\|^2_{\H}+\frac{\nu}{4}\|\nabla\mathscr{U}^k\|^2_{\H}\label{LC7}
	 	\end{align}
	 	and
	 	\begin{align}
	 		&\left|\left[\z^{-1}(t,\omega_k)-\z^{-1}(t,\omega_0)\right]b(\v^0,\mathscr{U}^k,\v^0)\right|\nonumber\\&\leq C\left|\z^{-1}(t,\omega_k)-\z^{-1}(t,\omega_0)\right|^2\|\v^0\|^2_{\H}\|\nabla\v^0\|^2_{\H}+\frac{\nu}{4}\|\nabla\mathscr{U}^k\|^2_{\H}.\label{LC8}
	 	\end{align}
	 	Combining \eqref{LC2}-\eqref{LC8}, we arrive at (replacing $t$ by $\zeta$)
	 	\begin{align}\label{LC13}
	 		\frac{\d }{\d \zeta}\|\mathscr{U}^k(\zeta)\|^2_{\H}+\frac{\sigma^2}{2}\|\mathscr{U}^k(\zeta)\|^2_{\H}+\nu\|\nabla\mathscr{U}^k(\zeta)\|^2_{\H}\leq
	 		P_1(\zeta)\|\mathscr{U}^k(\zeta)\|^2_{\H}+Q_1(\zeta),
	 	\end{align}
	 	for a.e. $\zeta\geq\tau$, where
	 	\begin{align*}
	 		P_1^k&=C\z^{-2}(\zeta,\omega_k)\|\nabla\v^0\|^2_{\H}+2\sigma |y(\theta_{\zeta}\omega_k)|,\\	Q_1^k&=C\left|\z(\zeta,\omega_k)-\z(\zeta,\omega_0)\right|^2\|\f\|^2_{\L^2(\R^2)}+C\left|y(\theta_{\zeta}\omega_k)-(\theta_{\zeta}\omega_0)\right|^2\|\v^0\|^2_{\H}\nonumber\\&\quad+C\left|\z^{-1}(\zeta,\omega_k)-\z^{-1}(\zeta,\omega_0)\right|^2\|\v^0\|^2_{\H}\|\nabla\v^0\|^2_{\H}.
	 	\end{align*}
	 	Now, from the fact that $\f\in\mathrm{L}^2_{\text{loc}}(\R;\L^2(\R^2))$ and $\v^0\in\mathrm{C}([\tau,+\infty);\H)\cap\mathrm{L}^2_{\mathrm{loc}}(\tau,+\infty;\V)$, we conclude that for all $t\ge\tau$
	 		 	\begin{align}\label{LC15}
	 		\lim_{k\to+\infty}\int_{\tau}^{t}P_1^k(\zeta)\d\zeta\leq C(t,\tau,\omega_0)\ \  \text{ and } \ \ 	\lim_{k\to+\infty}\int_{\tau}^{t}Q_1^k(\zeta)\d\zeta=0.
	 	\end{align}
	 	Making use of Gronwall's inequality in \eqref{LC13}, we infer for all $t\geq\tau$,
	 	\begin{align}\label{LC17}
	 		\|\mathscr{U}^k(t)\|^2_{\H}\leq e^{\int_{\tau}^{t}P_1^k(\zeta)\d\zeta}\left[\int_{\tau}^{t}Q_1^k(\zeta)\d\zeta\right].
	 	\end{align}
	 	In view of \eqref{LC15}-\eqref{LC17}, we find for all $t\geq\tau$,
	 	\begin{align}\label{LC18}
	 		\|\mathscr{U}^k(t)\|^2_{\H}\to0 \ \text{ as } \ k \to +\infty.
	 	\end{align}
		Moreover, \eqref{LC13} along with \eqref{LC15}-\eqref{LC17} imply 
	 	\begin{align}\label{LC19}
		\int_{\tau}^{t}\|\mathscr{U}^k(\zeta)\|^2_{\V}\d\zeta\to0 \ \text{ as } \ k\to +\infty,\ \text{ for all }\ t\geq\tau.
	 	\end{align}
	 
	 	Taking the inner product with $\A\mathscr{U}^k(\cdot)$ in \eqref{LC1}, we find 
	 	\begin{align}
	 		\frac{1}{2}\frac{\d }{\d t}\|\nabla\mathscr{U}^k\|^2_{\H}&=-\nu\|\A\mathscr{U}^k\|^2_{\H}-\left[\frac{\sigma^2}{2}-\sigma y(\theta_{t}\omega_k)\right]\|\nabla\mathscr{U}^k\|^2_{\H}-\z^{-1}(t,\omega_k)b(\mathscr{U}^k,\mathscr{U}^k,\A\mathscr{U}^k)\nonumber\\&\quad-\z^{-1}(t,\omega_k)b(\v^0,\mathscr{U}^k,\A\mathscr{U}^k)-\z^{-1}(t,\omega_k)b(\mathscr{U}^k,\v^0,\A\mathscr{U}^k)\nonumber\\&\quad-\left[\z^{-1}(t,\omega_k)-\z^{-1}(t,\omega_0)\right]b(\v^0,\v^0,\A\mathscr{U}^k)+\left[\z(t,\omega_k)-\z(t,\omega_0)\right](\f,\A\mathscr{U}^k).\label{LC2-V}
	 	\end{align}
	 	Using H\"older's and Young's inequalities, we obtain
	 	\begin{align}
	 		\left|\left[\z(t,\omega_k)-\z(t,\omega_0)\right](\f,\A\mathscr{U}^k)\right|&\leq C\left|\z(t,\omega_k)-\z(t,\omega_0)\right|^2\|\f\|^2_{\H}+\frac{\nu}{10}\|\A\mathscr{U}^k\|^2_{\H},\label{LC4-V}
	 		\\
	 		\left|\left[y(\theta_{t}\omega_k)-y(\theta_{t}\omega_0)\right](\v^0,\A\mathscr{U}^k)\right|&\leq C\left|y(\theta_{t}\omega_k)-(\theta_{t}\omega_0)\right|^2\|\v^0\|^2_{\H}+\frac{\nu}{10}\|\A\mathscr{U}^k\|^2_{\H}.\label{LC44-V}
	 	\end{align}
	 	Using \eqref{b1} and Young's inequality, we estimate
	 	\begin{align}
	 		&\left|\z^{-1}(t,\omega_k)b(\mathscr{U}^k,\mathscr{U}^k,\A\mathscr{U}^k)\right|\nonumber\\&\leq\frac{\nu}{10}\|\A\mathscr{U}^k\|^{2}_{\H}+C\|\mathscr{U}^k\|^{2}_{\H}+C\z^{-4}(t,\omega_k)\|\mathscr{U}^k\|^{4}_{\H}\|\nabla\mathscr{U}^k\|^{2}_{\H}\nonumber\\&\quad+C\z^{-2}(t,\omega_k)\|\mathscr{U}^k\|^{2}_{\H}\|\nabla\mathscr{U}^k\|^{2}_{\H}+C\z^{-4}(t,\omega_k)\|\mathscr{U}^k\|^{2}_{\H}\|\nabla\mathscr{U}^k\|^{4}_{\H},\label{LC7-V}	
	 		\end{align}
 		\begin{align}
	 		&\left|\z^{-1}(t,\omega_k)b(\v^0,\mathscr{U}^k,\A\mathscr{U}^k)+\z^{-1}(t,\omega_k)b(\mathscr{U}^k,\v^0,\A\mathscr{U}^k)\right|\nonumber\\&\leq \frac{\nu}{10} \|\A\mathscr{U}^k\|^2_{\H} + C\big\{\|\v^0\|^{2}_{\H}+\|\A\v^0\|^{2}_{\H}+\|\nabla\v^0\|^{2}_{\H}+\z^{-2}(t,\omega_k)\|\v^0\|^2_{\H}+\z^{-2}(t,\omega_k)\|\nabla\v^0\|^2_{\H}\nonumber\\&\quad+\z^{-4}(t,\omega_k)\|\v^0\|^2_{\H}\|\nabla\v^0\|^2_{\H}\big\}\|\mathscr{U}^k\|^2_{\H}+C\z^{-4}(t,\omega_k)\big\{\|\v^0\|^{2}_{\H}+\|\nabla\v^0\|^{2}_{\H}\nonumber\\&\quad+\|\v^0\|^2_{\H}\|\nabla\v^0\|^2_{\H}\big\}\|\nabla\mathscr{U}^k\|^{2}_{\H},\label{LC9-V}
	 	\end{align}
	 	and
	 	\begin{align}
	 	&	\left|\left[\z^{-1}(t,\omega_k)-\z^{-1}(t,\omega_0)\right]b(\v^0,\v^0,\A\mathscr{U}^k)\right|\nonumber\\&\leq C\left|\z^{-1}(t,\omega_k)-\z^{-1}(t,\omega_0)\right|^2\big\{\|\nabla\v^0\|^4_{\H}+\|\v^0\|^2_{\H}\|\nabla\v^0\|^2_{\H}+\|\v^0\|^2_{\H}\|\A\v^0\|^2_{\H}\nonumber\\&\quad+\|\nabla\v^0\|^2_{\H}\|\A\v^0\|^2_{\H}\big\}+\frac{\nu}{10}\|\A\mathscr{U}^k\|^2_{\H}.\label{LC10-V}
	 	\end{align}
	 	Combining \eqref{LC2-V}-\eqref{LC10-V}, we reach at (replacing $t$ by $\zeta$)
	 	\begin{align}\label{LC11-V}
	 		\frac{\d }{\d \zeta}\|\nabla\mathscr{U}^k(\zeta)\|^2_{\H}\leq
	 		\widehat{P}_1^k(\zeta)\|\nabla\mathscr{U}^k(\zeta)\|^2_{\H}+\widehat{Q}_1^k(\zeta),
	 	\end{align}
	 	for a.e. $\zeta\geq\tau$, where
	 	\begin{align*}
	 		\widehat{P}_1^k&=C\z^{-2}(\zeta,\omega_k)\|\mathscr{U}^k\|^{2}_{\H}+C\z^{-4}(\zeta,\omega_k)\big\{\|\mathscr{U}^k\|^{4}_{\H}+\|\mathscr{U}^k\|^{2}_{\H}\|\nabla\mathscr{U}^k\|^{2}_{\H}+\|\v^0\|^{2}_{\H}+\|\nabla\v^0\|^{2}_{\H}\nonumber\\&\quad+\|\v^0\|^2_{\H}\|\nabla\v^0\|^2_{\H}+2\sigma |y(\theta_{\zeta}\omega_k)|\big\},\\	\widehat{Q}_1^k&=C\big\{1+\|\v^0\|^{2}_{\H}+\|\A\v^0\|^{2}_{\H}+\|\nabla\v^0\|^{2}_{\H}+\z^{-2}(\zeta,\omega_k)\|\v^0\|^2_{\H}+\z^{-2}(\zeta,\omega_k)\|\nabla\v^0\|^2_{\H}\nonumber\\&\quad+\z^{-4}(\zeta,\omega_k)\|\v^0\|^2_{\H}\|\nabla\v^0\|^2_{\H}\big\}\|\mathscr{U}^k\|^2_{\H}+C\left|\z^{-1}(\zeta,\omega_k)-\z^{-1}(\zeta,\omega_0)\right|^2\big\{\|\nabla\v^0\|^4_{\H}\nonumber\\&\quad+\|\v^0\|^2_{\H}\|\nabla\v^0\|^2_{\H}+\|\v^0\|^2_{\H}\|\A\v^0\|^2_{\H}+\|\nabla\v^0\|^2_{\H}\|\A\v^0\|^2_{\H}\big\}\nonumber\\&\quad+C\left|\z(\zeta,\omega_k)-\z(\zeta,\omega_0)\right|^2\|\f\|^2_{\L^2(\R^2)}+C\left|y(\theta_{\zeta}\omega_k)-(\theta_{\zeta}\omega_0)\right|^2\|\v^0\|^2_{\H}.
	 	\end{align*}
 	Multiplying \eqref{LC11-V} by $(\zeta-\frac{\tau+t}{2})$ with $\zeta\in[\frac{\tau+t}{2},t]$ and making use of integration by parts as same as in \eqref{Conti9}, we obtain
 	\begin{align}\label{LC12-V}
 			\frac{t-\tau}{2}\|\nabla\mathscr{U}^k(t)\|^2_{\H}\leq\bigg[1+
 		\sup_{\zeta\in[\frac{\tau+t}{2},t]}\widehat{P}_1^k(\zeta)\bigg]\int_{\frac{\tau+t}{2}}^{t}\|\nabla\mathscr{U}^k(\zeta)\|^2_{\H}\d\zeta+\int_{\frac{\tau+t}{2}}^{t}\widehat{Q}_1^k(\zeta)\d\zeta.
 	\end{align}
 Using the fact $\f\in\mathrm{L}^{2}_{\emph{loc}}(\R;\L^2(\R^2))$ and $d_{\Omega_N}(\omega_k,\omega_0)\to0$, the estimates established in \eqref{ei2} and \eqref{ei8}, and convergences obtained in \eqref{LC18}-\eqref{LC19}, we conclude from \eqref{LC12-V} that $	\|\nabla\mathscr{U}^k(t)\|^2_{\H}\to0 \ \text{ as } \ k \to +\infty,$ for all $t>\tau$, which completes the proof along with \eqref{LC18}.
	 \end{proof}

Now, we are in a position to state and prove our main result of this work. 
\begin{theorem}\label{MainTheoRA}
	Suppose that Hypothesis \ref{Hyp-f} holds. Then the non-autonomous random DS $\Phi$ has a unique $(\H,\V)$-pullback $\mathfrak{D}$-random attractor $\mathscr{A}=\{\mathscr{A}(\tau,\omega): \tau\in\R \text{ and }\omega\in \Omega\}\in\mathfrak{D}$ given by
		\begin{align*}
		\mathscr{A}(\tau,\omega)&=\bigcap_{s>0}\overline{\bigcup_{t\geq s}\Phi(t,\tau-t,\theta_{-t}\omega,\mathcal{K}(\tau-t,\theta_{-t}\omega))}^{\H}\\&=\bigcap_{s>0}\overline{\bigcup_{t\geq s}\Phi(t,\tau-t,\theta_{-t}\omega,\mathcal{K}(\tau-t,\theta_{-t}\omega))}^{\mathbb{V}},
	\end{align*}
where $\mathcal{K}(\tau,\omega)$ is the absorbing set obtained in Lemma \ref{PA1}.
\end{theorem}
\begin{proof}
	Since $\Phi$ is $(\H,\V)$-continuous random cocycle (see Lemma \ref{Continuity}) and $\mathfrak{D}$ is an inclusion closed universe, we apply the abstract result from the work \cite{WZhao} (see Theorem \ref{AbstractResult} above) to prove this theorem. Lemma \ref{PA1} shows that $\Phi$ has closed pullback $\mathfrak{D}$-random absorbing set, Lemma \ref{PA2} reveals that $\Phi$ is $(\H,\H)$-pullback asymptotically compact and Lemma \ref{LusinC} proves that $\Phi(t,\tau,\cdot,\u_{\tau}):\Omega\to\mathbb{V}$ is $(\mathscr{F},\mathscr{B}(\V))$-measurable for every fixed $t>0$, $\tau\in\R$ and $\u_{\tau}\in\H$. We conclude that all the three conditions stated in Theorem \ref{AbstractResult} for $(\H,\V)$-continuous random cocycle $\Phi$ are satisfied. Hence, an application of Theorem \ref{AbstractResult} completes the proof.
\end{proof}
\begin{remark}\label{rem3.14}
	We observe that the method introduced in \cite{Wang} for the upper semicontinuity is not suitable for proving the upper semicontinuity of random attractors obtained in Theorem \ref{MainTheoRA} as $\sigma\to 0$. Because the right hand side of \eqref{ue} will tend to $\infty$ as $\frac{\sigma^2}{2}\to0$ and we will not able to find a random set which contains $\bigcup\limits_{0<\sigma\leq1}\mathcal{K}_{\sigma}(\tau,\omega)$, where $\mathcal{K}_{\sigma}(\tau,\omega)$ represents the absorbing set corresponding to each $\sigma$. Therefore the existence of pullback attractors for 2D deterministic NSE as well as the upper semicontinuity of random attractors for 2D SNSE on the whole space are still challenging open problems. 
 \end{remark}

\begin{remark}\label{rem315}
	In \cite{KM2}, the upper semicontinuity of the random attractors with respect to domains, that is, when domain changes from bounded to unbounded (Poincar\'e) domain is proved for stochastic convective Brinkman-Forchheimer equations (cf. \cite[Theorem 6.10]{KM2}). The authors have  also discussed the upper semicontinuity of the random attractors with respect to domains for  SNSE driven by additive noise as a remark (cf. \cite[Remark 6.11]{KM2}). Using the similar arguments as in the work \cite{KM2}, one can establish the upper semicontinuity of the random attractors with respect to domains for the 2D non-autonomous SNSE \eqref{SNSE} on the whole space.
\end{remark}

\begin{remark}\label{rem3.15}
	The authors in \cite{RKM} proved the asymptotic autonomy of pullback random attractors for 2D SNSE on unbounded Poincar\'e domains. They consider the additive as well as linear multiplicative white noise where the stochastic integration has been taken in the sense of Stratonovich. Using the similar ideas as in the work \cite{RKM}, one can obtain the asymptotic autonomy robustness of random attractors for the 2D non-autonomous SNSE \eqref{SNSE}. It means that if $\mathscr{A}_{\infty}(\omega)$ is the unique random attractor for the 2D autonomous SNSE \eqref{SNSE}, then one can prove that 
		 \begin{align}\label{AA}
		\lim_{\tau\to-\infty}\mathrm{dist}_{\L^2(\R^2)}\left(\mathscr{A}(\tau,\omega),\mathscr{A}_{\infty}(\omega)\right)=0,\  \text{ for all }\ \omega\in\Omega,
	\end{align}
	where $\mathrm{dist}_{\L^2(\R^2)}(\cdot,\cdot)$ denotes the Hausdorff semi-distance between two non-empty subsets of $\L^2(\R^2)$ (cf. \cite{RKM} for the detailed proof).
\end{remark}

	\section{Invariant Sample Measures and A Stochastic Liouville Type Theorem}\label{sec4}\setcounter{equation}{0}
In this section, we demonstrate the existence of a family of invariant sample measures which satisfies a stochastic Liouville type theorem (cf. \cite{CY,ZWC}) for the 2D non-autonomous SNSE \eqref{SNSE} on the whole space $\R^2$. 

\subsection{Invariant sample measures}
    Let us recall some definitions and results from \cite{Arnold}. Let $\X$ be a Polish space and $\Phi$ be a random DS over $\theta$.
\begin{definition}[Skew-product, \cite{Arnold}]\label{SPF}
	Given a random DS $\Phi$, the mapping
	\begin{align}
		\Theta_t: \Omega \times \X \ni (\omega,x) \mapsto (\theta_t(\omega),\Phi(t,\omega,x)) = \Theta_t(\omega,x) \in \Omega \times \X, \ \ t\in\R^+,
	\end{align}
	is a measurable DS on $(\Omega\times\X,\mathscr{F}\otimes\mathscr{B}(\X))$ which is called the \emph{skew product} of metric DS $(\Omega,\mathscr{F},\mathbb{P},\{\theta_{t}\}_{t\in\R})$ and random DS $\Phi(t,\omega)$ on $\X$.
\end{definition}
\begin{definition}[Invariant measure for $\Theta$, \cite{Arnold}]\label{IMSP}
	A probability measure $\uprho$ on $(\Omega\times\X,\mathscr{F}\otimes\mathscr{B}(\X))$ is called \emph{invariant for $\Theta$} corresponding to $\Phi$, if $\Theta_t\uprho=\uprho$ for all $t\in\R^+$.
\end{definition}
\begin{definition}[Invariant measure for $\Phi$, \cite{Arnold}]\label{IMRDS}
	A probability measure $\uprho$ on $(\Omega\times\X,\mathscr{F}\otimes\mathscr{B}(\X))$ is called \emph{invariant for $\Phi$} ($\Phi$-invariant), if 
	\begin{itemize}
		\item [(i)] $\Theta_t\uprho=\uprho$ for all $t\in\R^+$,
		\item [(ii)] $\pi_{\Omega}\uprho=\mathbb{P}$, that is, the first marginal of $\uprho$ is $\mathbb{P}$, where $\pi_{\Omega}:\Omega\times\X\ni(\omega,x)\mapsto\omega\in\Omega$.
	\end{itemize} 
\end{definition}
\noindent
Define
\begin{align*}
	&\mathscr{P}_{\mathbb{P}}(\Omega\times\X)\nonumber\\&:=\{\uprho:\uprho \text{ is probability measure on } (\Omega\times\X,\mathscr{F}\otimes\mathscr{B}(\X)) \text{ with marginal } \mathbb{P} \text{ on } (\Omega,\mathscr{F})\}.
\end{align*}
\begin{definition}[Sample measure, \cite{Arnold}]\label{SM}
	Let $\uprho\in\mathscr{P}_{\mathbb{P}}(\Omega\times\X)$. A mapping $\uprho_{\cdot}(\cdot):\Omega\times\mathscr{B}(\X)\to[0,1]$ is called a \emph{sample measure (or disintegration)} of $\uprho$ with respect to $\mathbb{P}$ if
	\begin{itemize}
		\item [(i)] for all $E\in\mathscr{B}(\X)$, $\omega\mapsto\uprho_{\omega}(E)$ is $\mathscr{F}$-measurable,
		\item [(ii)]  for $\mathbb{P}$-a.s. $\omega\in\Omega$, $E\mapsto\uprho_{\omega}(E)$ is a probability measure on $(\X,\mathscr{B}(\X))$,
		\item[(iii)] for all $G\in\mathscr{F}\times\mathscr{B}(\X)$,
		\begin{align}
			\uprho(G)=\int_{\Omega}\int_{\X}1_{G}(\omega,x)\uprho_{\omega}(\d x)\mathbb{P}(\d\omega).
		\end{align}
	\end{itemize}
\end{definition}

Let us  now recall some basic definitions from the work \cite{CY}.
\begin{definition}[Invariant sample measure, \cite{CY,ZWC}]\label{D-ISM}
	Let $\mathscr{P}(\X)$ denote the space of all probability measures on $\H$. A mapping $(t,\omega)\in\R\times\Omega\mapsto\uprho_{t,\omega}\in\mathscr{P}(\X)$ is called an invariant sample measure for continuous cocycle $\Phi$ if 
	\begin{align*}
		\int_{\X}\varphi(\u)\uprho_{t,\omega}(\d\u)&=\int_{\X}\varphi(\Phi(-\xi,t+\xi,\theta_{\xi}\omega,\u))\uprho_{t+\xi,\theta_{\xi}\omega}(\d\u),
	\end{align*} 
	for any real-valued continuous functional $\varphi$ on $\X$.
\end{definition}

\begin{definition}[Generalized Banach limit, \cite{CY,ZWC}]
	A  \emph{generalized Banach limit} is any linear functional, denoted by $\underset{t\to+\infty}{\mathrm{LIM}}$, defined on the space of all bounded real-valued functions on $[0,+\infty)$ and satisfying
	\begin{itemize}
		\item [(i)]$\underset{t\to+\infty}{\mathrm{LIM}}\zeta(t)\geq0$ for non-negative functions $\zeta(\cdot)$ on $[0,+\infty)$,
		\item [(ii)] $\underset{t\to+\infty}{\mathrm{LIM}}\zeta(t)=\underset{t\to+\infty}{\lim}\zeta(t)$ if the usual limit $\underset{t\to+\infty}{\lim}\zeta(t)$ exists.
	\end{itemize}
\end{definition}

In order to apply the abstract results from the work \cite{CY}, we require that $\Phi$ has a unique $\mathfrak{D}$-pullback random attractor, and for given $\u^*\in\H$, $t\in\R$ and $\omega\in\Omega$, the mapping $(\xi,\u^*)\mapsto\Phi(-\xi,t+\xi,\theta_{\xi}\omega,\u^*)$ is continuous from $(-\infty,0]\times\H$ to $\H$. The existence of $\mathfrak{D}$-pullback random attractors is shown in Theorem \ref{MainTheoRA}. The following lemma is helpful to obtain the continuity (discussed above) of the mapping $\Phi$.
\begin{lemma}\label{R_continuity}
	For given $\u^*\in\H$, $t\in\R$ and $\omega\in\Omega$, the mapping $\xi\mapsto\Phi(t-\xi,\xi,\theta_{\xi}\omega,\u^*)$ is right continuous from $(-\infty,t]$ to $\mathbb{H}$.	
\end{lemma}
\begin{proof}
	Since $\Phi(t-\xi,\xi,\theta_{\xi}\omega,\u^*)=\u(t;\xi,\omega,\u^*)=\v(t;\xi,\omega,\v^*)\z^{-1}(t,\omega)$, it is enough to prove that $\v(t;\cdot,\omega,\v^*)$ with $\v^*=\u^*\z(\cdot,\omega)$ is right continuous on $(-\infty,t]$. Let us fix $\xi^*\in\R$ and $\v^*\in\H$, $\omega\in\Omega$. Now, we only need to prove that for any given $\varepsilon>0$, we can find a positive real number $\varepsilon^*=\varepsilon^*(\varepsilon,\xi^*,\omega,\v^*)$ such that 
	\begin{align}\label{RLTT1}
		\|\v(\tau;\xi,\omega,\v^*)-\v^*\|_{\H}<\varepsilon, \ \text{ whenever }\  \xi\in(\xi^*,\xi^*+\varepsilon^*), \ \tau\in(\xi,\xi^*+\varepsilon^*),
	\end{align}
	where $\v(\tau;\xi,\omega,\v^*)$ is the solution of the system \eqref{CNSE} with the initial data $\v^*$ and initial time $\xi$. 
	
	We have 
	\begin{align}\label{RLTT2}
		\|\v(\tau;\xi,\omega,\v^*)-\v^*\|^2_{\H}&=\|\v(\tau;\xi,\omega,\v^*)\|^2_{\H}-\|\v^*\|^2_{\H}-2(\v(\tau;\xi,\omega,\v^*)-\v^*,\v^*)\nonumber\\&=\int_{\xi}^{\tau}\frac{\d}{\d\zeta}\|\v(\zeta;\xi,\omega,\v^*)\|^2_{\H}\d\zeta-2(\v(\tau;\xi,\omega,\v^*)-\v^*,\v^*).
	\end{align}
	We infer from \eqref{ei1} that 
	\begin{align}\label{RLTT3}
		&\int_{\xi}^{\tau}\frac{\d}{\d\zeta}\|\v(\zeta;\xi,\omega,\v^*)\|^2_{\H}\d\zeta\nonumber\\&\leq \frac{2}{\sigma^2}\int_{\xi}^{\tau}\z^2(\zeta,\omega)\|\f(\zeta)\|^2_{\L^2(\R^2)}\d\zeta+2 \sigma\int_{\xi}^{\tau}|y(\theta_{\zeta}\omega)|\|\v(\zeta)\|^2_{\H}\d\zeta<+\infty.
	\end{align}
	From \eqref{ei22}, we find
	\begin{align}\label{RLTT4}
		&\sup_{\xi\in[\xi^*-1,\xi^*+1]}\|\v(\zeta;\xi,\omega,\v^*)\|^2_{\H}\nonumber\\&\leq\sup_{\xi\in[\xi^*-1,\xi^*+1]}\bigg[\|\v^*\|^2_{\H}+\frac{2}{\sigma^2}\int_{\xi}^{\zeta}\z^2(r,\omega)\|\f(r)\|^2_{\L^2(\R^2)}\d r\bigg]e^{2\sigma\int_{\xi}^{\zeta}| y(\theta_{r}\omega)|\d r}\nonumber\\&\leq\left[\|\v^*\|^2_{\H}+\frac{2}{\sigma^2}\int_{\xi^*-1}^{\zeta}\z^2(r,\omega)\|\f(r)\|^2_{\L^2(\R^2)}\d r\right]e^{2\sigma\int_{\xi^*-1}^{\zeta}| y(\theta_{r}\omega)|\d r},
	\end{align}
	which is finite and the right hand side is independent of $\xi$. It implies from \eqref{RLTT3} that we can find a positive real number $\varepsilon^*_1=\varepsilon^*_1(\varepsilon,\xi^*,\omega,\v^*)$ such that
	\begin{align}\label{RLTT5}
		\int_{\xi}^{\tau}\frac{\d}{\d\zeta}\|\v(\zeta;\xi,\omega,\v^*)\|^2_{\H}\d\zeta\leq \frac{\varepsilon^2}{2},  \ \text{ whenever }\  \xi\in(\xi^*,\xi^*+\varepsilon^*_1), \ \tau\in(\xi,\xi^*+\varepsilon^*_1).
	\end{align}
	Now we estimate the final term on the right hand side of \eqref{RLTT2}. From the fact that $$\v(\cdot;\xi,\omega,\v^*)\in\mathrm{C}([\xi,+\infty);\H)\cap\mathrm{L}^2_{\mathrm{loc}}(\xi,+\infty;\V),$$ and \eqref{RLTT4}, we can find a positive number $K(\omega,\xi^*,\v^*)$ which is independent of $\xi$ such that 
	\begin{align}\label{RLTT6}
		\max_{\zeta\in[\xi^*-1,\xi^*+1]}\|\v(\zeta;\xi,\omega,\v^*)\|^2_{\H}\leq K(\omega,\xi^*,\v^*), \ \text{ for all }\ \xi\in[\xi^*-1,\zeta].
	\end{align}
	Since $\V$ is dense in $\H$, for $\varepsilon>0$ same as in \eqref{RLTT1}, we can find an element $\tilde{\v}\in\V$ such that 	\begin{align}\label{RL}\|\tilde{\v}-\v^*\|_{\H}\leq\frac{\varepsilon^2}{8(K(\omega,\xi^*,\v^*)+\|\v^*\|_{\H})}.\end{align}
	Thus, for $\xi\in(\xi^*,\xi^*+\varepsilon^*_1)$ and $\tau\in(\xi,\xi^*+\varepsilon^*_1)$, in view of \eqref{RLTT6} and \eqref{RL}, we obtain
	\begin{align}\label{RLTT7}
		|(\v(\tau;\xi,\omega,\v^*)-\v^*,\v^*)|&\leq|(\v(\tau;\xi,\omega,\v^*)-\v^*,\tilde{\v})|+|(\v(\tau;\xi,\omega,\v^*)-\v^*,\tilde{\v}-\v^*)|\nonumber\\&\leq|(\v(\tau;\xi,\omega,\v^*)-\v^*,\tilde{\v})|+\frac{\varepsilon^2}{8}.
	\end{align}
	Now, we consider
	\begin{align}\label{RLTT8}
		|(\v(\tau;\xi,\omega,\v^*)-\v^*,\tilde{\v})|&=\left|\left\langle\int_{\xi}^{\tau}\frac{\d}{\d\zeta}\v(\zeta;\xi,\omega,\v^*)\d\zeta,\tilde{\v}\right\rangle\right|\nonumber\\&\leq\|\tilde{\v}\|_{\V}\bigg(\int_{\xi}^{\tau}\bigg\|\frac{\d}{\d\zeta}\v(\zeta;\xi,\omega,\v^*)\bigg\|^2_{\V^*}\d\zeta\bigg)^{1/2}(\tau-\xi)^{1/2}.
	\end{align}
	From the first equation of the system \eqref{CNSE},  for any $\bar{\v}\in\V$, we have 
	\begin{align*}
		&\left|	\left	\langle\frac{\d}{\d\zeta}\v(\zeta;\xi,\omega,\v^*),\bar{\v}\right\rangle\right|\nonumber\\&\leq\nu\|\nabla\v(\zeta;\xi,\omega,\v^*)\|_{\H}\|\nabla\bar{\v}\|_{\H}+\left[\frac{\sigma^2}{2}+2\sigma|y(\theta_{\zeta}\omega)|\right]\|\v(\zeta;\xi,\omega,\v^*)\|_{\H}\|\bar{\v}\|_{\H}\nonumber\\&\quad+\sqrt{2}\z^{-1}(\zeta,\omega)\|\v(\zeta;\xi,\omega,\v^*)\|_{\H}\|\nabla\v(\zeta;\xi,\omega,\v^*)\|_{\H}\|\nabla\bar{\v}\|_{\H}+\z(\zeta,\omega)\|\f(\zeta)\|_{\L^2(\R^2)}\|\bar{\v}\|_{\H},\nonumber\\&\leq C\big[\|\nabla\v(\zeta;\xi,\omega,\v^*)\|_{\H}+\|\v(\zeta;\xi,\omega,\v^*)\|_{\H}+|y(\theta_{\zeta}\omega)|\|\v(\zeta;\xi,\omega,\v^*)\|_{\H}\nonumber\\&\quad+\z^{-1}(\zeta,\omega)\|\v(\zeta;\xi,\omega,\v^*)\|_{\H}\|\nabla\v(\zeta;\xi,\omega,\v^*)\|_{\H}+\z(\zeta,\omega)\|\f(\zeta)\|_{\L^2(\R^2)}\big]\|\bar{\v}\|_{\V},
	\end{align*}
	which gives
	\begin{align}\label{RLTT9}
		\bigg\|\frac{\d}{\d\zeta}\v(\zeta;\xi,\omega,\v^*)\bigg\|^2_{\V^*}&\leq C\big[\|\v(\zeta;\xi,\omega,\v^*)\|^2_{\V}+\z^{-2}(\zeta,\omega)\|\v(\zeta;\xi,\omega,\v^*)\|^2_{\H}\|\v(\zeta;\xi,\omega,\v^*)\|^2_{\V}\nonumber\\&\qquad+|y(\theta_{\zeta}\omega)|^2\|\v(\zeta;\xi,\omega,\v^*)\|^2_{\H}+\z^2(\zeta,\omega)\|\f(\zeta)\|^2_{\L^2(\R^2)}\big].
	\end{align}
	Combining \eqref{RLTT8} and \eqref{RLTT9}, we infer
	\begin{align}\label{RLTT10}
		&	|(\v(\tau;\xi,\omega,\v^*)-\v^*,\tilde{\v})|\nonumber\\&\leq\bigg(\int_{\xi}^{\tau}\big[\|\v(\zeta;\xi,\omega,\v^*)\|^2_{\V}+\z^{-2}(\zeta,\omega)\|\v(\zeta;\xi,\omega,\v^*)\|^2_{\H}\|\v(\zeta;\xi,\omega,\v^*)\|^2_{\V}\nonumber\\&\qquad+|y(\theta_{\zeta}\omega)|^2\|\v(\zeta;\xi,\omega,\v^*)\|^2_{\H}+\z^2(\zeta,\omega)\|\f(\zeta)\|^2_{\L^2(\R^2)}\big]\d\zeta\bigg)^{1/2}\|\tilde{\v}\|_{\V}(\tau-\xi)^{1/2}.
	\end{align}
	From the fact that  $\v(\cdot;\xi,\omega,\v^*)\in\mathrm{C}([\xi,+\infty);\H)\cap\mathrm{L}^2_{\mathrm{loc}}(\xi,+\infty;\V)$, continuity of $\z(\cdot,\omega)$, \eqref{RLTT4} and \eqref{RLTT10}, it follows that for $\varepsilon>0$ same as in \eqref{RLTT1}, we can find a positive real number $\varepsilon^*_2=\varepsilon^*_2(\varepsilon,\xi^*,\omega,\v^*)$ such that
	\begin{align}\label{RLTT11}
		|(\v(\tau;\xi,\omega,\v^*)-\v^*,\tilde{\v})|\leq \frac{\varepsilon^2}{8},  \ \text{ whenever }\  \xi\in(\xi^*,\xi^*+\varepsilon^*_2), \ \tau\in(\xi,\xi^*+\varepsilon^*_2).
	\end{align}
	Taking $\varepsilon^*=\min\{\varepsilon^*_1,\varepsilon^*_2\}$, we obtain \eqref{RLTT1} by combining \eqref{RLTT2}, \eqref{RLTT5}, \eqref{RLTT7} and \eqref{RLTT11}.
\end{proof}
One can prove that, for given $\u^*\in\H$, $t\in\R$ and $\omega\in\Omega$, the mapping $\xi\mapsto\Phi(t-\xi,\xi,\theta_{\xi}\omega,\u^*)$ is left continuous from $(-\infty,t]$ to $\mathbb{H}$ by following the analogous steps as in the proof of Lemma \ref{R_continuity}. 
 We provide the following result on the continuity of the $\mathbb{H}$-valued mapping $\xi\mapsto\Phi(t-\xi,\xi,\theta_{\xi}\omega,\u^*)$ on $(-\infty,t]$.
\begin{lemma}\label{Cbdd}
	For given $\u^*\in\H$, $t\in\R$ and $\omega\in\Omega$, the mapping $\xi\mapsto\Phi(-\xi,t+\xi,\theta_{\xi}\omega,\u^*)$ is continuous from $(-\infty,0]$ to $\mathbb{H}$.
\end{lemma}
Now combining the results of Lemmas \ref{Continuity} and \ref{Cbdd}, we state the following result:
\begin{lemma}\label{ContinuityBdd}
	For given $\u^*\in\H$, $t\in\R$ and $\omega\in\Omega$, the $\mathbb{H}$-valued mapping $$(\xi,\u^*)\mapsto\Phi(-\xi,t+\xi,\theta_{\xi}\omega,\u^*)$$ is continuous and bounded on $(-\infty,0]\times\H$.
\end{lemma}
In light of the abstract result on invariant sample measures which was established in \cite[Theorem 3.1]{CY}, we can say (using Theorem \ref{MainTheoRA} and Lemma \ref{ContinuityBdd}) that the 2D SNSE \eqref{SNSE} has a family of invariant sample measures $\{\uprho_{t,\omega}\}_{t\in\R,\omega\in\Omega}$ for $\Phi$ and is supported on $\mathscr{A}(t,\omega)$. Therefore we have the following result:
\begin{theorem}\label{ISM}
	Assume that Hypothesis \ref{Hyp-f} holds and let $\Phi$ be the non-autonomous random DS generated by the 2D SNSE \eqref{SNSE} over the metric DS $(\Omega,\mathscr{F},\mathbb{P},\{\theta_{t}\}_{t\in\R})$ with the state space $\H$. Then for a given generalized Banach limit $\underset{t\to\infty}{\rm LIM}$ and a given continuous function $\upsilon(\cdot):\R\to\H$ satisfying $\{\upsilon(\tau)\}_{\tau\in\R}\in\mathfrak{D}$, there exists a family of Borel probability measures $\{\uprho_{\tau,\omega}\}_{\tau\in\R,\omega\in\Omega}$ such that $\uprho_{\tau,\omega}$ is an invariant sample measure for $\Phi$ and is supported on $\mathscr{A}(\tau,\omega)$, and 
	\begin{align}\label{ISM1}
		&\underset{t\to-\infty}{\rm LIM}\frac{1}{-t}\int_{t}^{0}\varphi(\Phi(-\xi,\tau+\xi,\theta_{\xi}\omega,\upsilon(\tau+\xi)))\d\xi\nonumber\\&=\int_{\H}\varphi(\u)\uprho_{\tau,\omega}(\d\u)=\int_{\mathscr{A}(\tau,\omega)}\varphi(\u)\uprho_{\tau,\omega}(\d\u)\nonumber\\&=\underset{t\to-\infty}{\rm LIM}\frac{1}{-t}\int_{t}^{0}\left[\int_{\mathscr{A}(\tau+\xi,\theta_{\xi}\omega)}\varphi(\Phi(-\xi,\tau+\xi,\theta_{\xi}\omega,\u))\uprho_{\tau+\xi,\omega}(\d\u)\right]\d\xi,
	\end{align}
	for any real-valued continuous functional $\varphi$ on $\H$. 
	For all $\omega\in\Omega$, $\uprho_{\theta_{t}\omega}$ is invariant with respect to the non-autonomous random DS $\Phi$ in the sense that 
	\begin{align}\label{ISM2}
		\int_{\mathscr{A}(t,\omega)}\varphi(\u)\uprho_{\theta_{t}\omega}(\d\u)=\int_{\mathscr{A}(\tau,\theta_{-(t-\tau)}\omega)}\varphi(\Phi(t-\tau,\tau,\theta_{\tau-t}\omega,\u))\uprho_{\tau,\theta_{-(t-\tau)}\omega}(\d\u), \text{ for all } t\geq \tau.
	\end{align}
\end{theorem}
\subsection{A stochastic Liouville type theorem} 
Our next aim is to study \emph{a stochastic Liouville type theorem} for the 2D non-autonomous SNSE \eqref{SNSE}. For this purpose, we require the definition of a class of test functions. Let us rewrite the first equation of the system \eqref{SNSE} as 
\begin{align}
	\d\u=L(\u,t)\d t+\sigma\u\d\W=\left[-\nu \A\u-\B(\u)+\mathcal{P}\f\right]\d t +\sigma\u\d\W.
\end{align} 
\subsubsection{Class of test function and its properties} Let us define the class $\mathscr{J}$ of cylindrical test functions $\Lambda:\H\to\R$ of the form 
\begin{align*}
	\Lambda(\u):=\psi((\u,e_1),(\u,e_2),\ldots,(\u,e_m)),
\end{align*}
where $\psi$ is a two times differentiable scalar valued function with compact support defined on $\R^m$, $m\in\N$, and $\{e_1,e_2,\ldots,e_m\}\subset\mathcal{V}_2:=\{\u\in\mathrm{C}^2_{c}(\R^2;\R^2):\nabla\cdot\u=0\}$, where $\mathrm{C}^2_{c}(\R^2;\R^2)$ is two times differentiable $\R^2$-valued function with compact support (see \cite{CR} for detailed information of such types of test functions). 
Moreover, the chain rule of differential calculus gives
\begin{align}
	\Lambda'(\u)&=\sum_{i=1}^{m}\partial_{i}\psi((\u,e_1),(\u,e_2),\ldots,(\u,e_m))e_{i},\\
	\Lambda''(\u)&=\sum_{i,j=1}^{m}\partial_{i,j}\psi((\u,e_1),(\u,e_2),\ldots,(\u,e_m)){e_{i}\cdot e_j},
\end{align}
where $\partial_{i}\psi$ and $\partial_{i,j}\psi$ are the first-order and second-order derivatives of $\psi$ with respect to $i$-th variable, and $i$-th and $j$-th variables, respectively.
Now, let $\u(\cdot)$ be  the solution of \eqref{SNSE}. Then, we infer from It\^o's formula (cf. \cite[Theorem 2.4]{ZBNJ}) that for $t\geq\tau$ and $\omega\in\Omega$,
\begin{align}\label{ItoF}
	\Lambda(\u(t))-\Lambda(\u(\tau))&=\int_{\tau}^{t}\left(\Lambda'(\u(\xi)),\sigma\u(\xi)\right)\d\W(\xi)+\int_{\tau}^{t}\left\langle\Lambda'(\u(\xi)),L(\u(\xi),\xi)\right\rangle\d\xi\nonumber\\&\quad+\frac{1}{2}\int_{\tau}^{t}\Lambda''(\u(\xi))(\sigma\u(\xi))(\sigma\u(\xi))\d\xi.
\end{align}
Next lemma is useful to prove that the family of invariant sample measures obtained in Theorem \ref{ISM} satisfies a stochastic Liouville type theorem.
\begin{lemma}\label{G-Con}
	Let $\Lambda\in\mathscr{J}$, $\u\in\H$ and $\f\in\L^2(\R^2)$, then the mappings $\mathrm{G}_i:\H\to\R$ (for $i\in\{1,2,3\}$) given by 
	\begin{align*}
		\mathrm{G}_1(\u)=\left\langle L(\u,t),\Lambda'(\u)\right\rangle, \ \  \mathrm{G}_2(\u)=\big(\sigma\u,\Lambda'(\u)\big) \ \text{ and } \ \mathrm{G}_3(\u)=\Lambda''(\u)(\sigma\u)(\sigma\u),
	\end{align*}
are well-defined and continuous in $\H$.
\end{lemma}
\begin{proof}
	Since $\{e_1,\ldots,e_m\}\subset\mathcal{V}_2,$ we have $\Lambda'(\u),\ \Lambda''(\u)\in\mathcal{V}_2$. Integration by parts gives 
	\begin{align*}
		\mathrm{G}_1(\u)&=\left\langle L(\u,t),\Lambda'(\u)\right\rangle\nonumber\\&=\left\langle-\nu \A\u-\B(\u)+\mathcal{P}\f ,\Lambda'(\u)\right\rangle\nonumber\\&=\nu\left(\u,\Delta\Lambda'(\u)\right)+b(\u,\Lambda'(\u),\u)+\left(\f ,\Lambda'(\u)\right)\nonumber\\&\leq \nu\|\u\|_{\H}\|\Delta\Lambda'(\u)\|_{\L^2(\R^2)}+\|\nabla\Lambda'(\u)\|_{\L^{\infty}(\R^2)}\|\u\|^2_{\H}+\|\f\|_{\L^2(\R^2)}\|\Lambda'(\u)\|_{\H}<\infty.
	\end{align*}
Thus, $\mathrm{G}_1(\u)$ is well-defined. Now consider
\begin{align}
	&|\mathrm{G}_1(\u_1)-\mathrm{G}_1(\u_2)|\nonumber\\&=\left\langle-\nu \A\u_1-\B(\u_1)+\mathcal{P}\f ,\Lambda'(\u_1)\right\rangle-\left\langle-\nu \A\u_2-\B(\u_2)+\mathcal{P}\f ,\Lambda'(\u_2)\right\rangle\nonumber\\&=-\nu\left\langle\A\u_1-\A\u_2,\Lambda'(\u_1)\right\rangle-\nu\left\langle\A\u_2,\Lambda'(\u_1)-\Lambda'(\u_2)\right\rangle -\left\langle\B(\u_1)-\B(\u_2),\Lambda'(\u_1)\right\rangle\nonumber\\&\quad-\left\langle\B(\u_2),\Lambda'(\u_1)-\Lambda'(\u_2)\right\rangle+\left(\mathcal{P}\f,\Lambda'(\u_1)-\Lambda'(\u_2)\right)\nonumber\\&=\nu\left(\u_1-\u_2,\Delta\Lambda'(\u_1)\right)+\nu\left(\u_2,\Delta\Lambda'(\u_1)-\Delta\Lambda'(\u_2)\right)+b(\u_1-\u_2,\Lambda'(\u_1),\u_1)\nonumber\\&\quad+b(\u_1,\Lambda'(\u_1),\u_1-\u_2)+b(\u_2,\Lambda'(\u_1)-\Lambda'(\u_2),\u_2)+\left(\f,\Lambda'(\u_1)-\Lambda'(\u_2)\right)\nonumber\\&\leq \nu\|\u_1-\u_2\|_{\H}\|\Delta\Lambda'(\u_1)\|_{\H}+\nu\|\u_2\|_{\H}\|\Delta(\Lambda'(\u_1)-\Lambda'(\u_2))\|_{\H}\nonumber\\&\quad+\|\u_1-\u_2\|_{\H}\|\nabla\Lambda'(\u_1)\|_{\L^{\infty}(\R^2)}\|\u_1\|_{\H}+\|\u_1\|_{\H}\|\nabla\Lambda'(\u_1)\|_{\L^{\infty}(\R^2)}\|\u_1-\u_2\|_{\H}\nonumber\\&\quad+\|\u_2\|_{\H}\|\nabla(\Lambda'(\u_1)-\Lambda'(\u_2))\|_{\L^{\infty}(\R^2)}\|\u_2\|_{\H}+\|\f\|_{\L^2(\R^2)}\|\Lambda'(\u_1)-\Lambda'(\u_2)\|_{\H}.
\end{align}
Then by the continuity and boundedness of $\Lambda'(\cdot)\in\mathcal{V}_2$, we conclude that the mapping $\u\mapsto\left\langle L(\u,t),\Lambda'(\u)\right\rangle$ is continuous in $\H$. Similarly, one can obtain that the mappings $\u\mapsto\big(\sigma\u,\Lambda'(\u)\big)$  and $\u\mapsto\Lambda''(\u)(\sigma\u)(\sigma\u)$ are well-defined and continuous on $\H$.
\end{proof}
Next, we provide a stochastic Liouville type theorem for the 2D non-autonomous SNSE \eqref{SNSE} on the whole space.
\begin{theorem}\label{SLTT}
	Let the Hypothesis \ref{Hyp-f} hold. Then, the family of invariant sample measures obtained in Theorem \ref{ISM} satisfies a stochastic Liouville type equation, that is, for all $t\geq\tau$ and $\omega\in\Omega$,
	\begin{align}\label{SLTT1}
		&\int_{\mathscr{A}(t,\omega)}\Lambda(\u^*)\uprho_{t,\omega}(\d\u^*)-\int_{\mathscr{A}(\tau,\theta_{-(t-\tau)}\omega)}\Lambda(\u^*)\uprho_{\tau,\theta_{-(t-\tau)}\omega}(\d\u^*)\nonumber\\&=\int_{\tau}^{t}\int_{\H}\left(\Lambda'(\u^*),\sigma\u^*\right)\uprho_{\xi,\theta_{-(t-\xi)}\omega}(\d\u^*)\d\widetilde{\W}(\xi)+\int_{\tau}^{t}\int_{\H}\left\langle\Lambda'(\u^*),L(\u^*,\xi)\right\rangle\uprho_{\xi,\theta_{-(t-\xi)}\omega}(\d\u^*)\d \xi \nonumber\\&\quad+\frac{1}{2}\int_{\tau}^{t}\int_{\H}\left(\Lambda''(\u^*)(\sigma\u^*)(\sigma\u^*)\right)\uprho_{\xi,\theta_{-(t-\xi)}\omega}(\d\u^*)\d \xi,
	\end{align}
where $\sigma>0$, $\Lambda\in\mathscr{J}$, and $\widetilde{\W}(\xi)=\W(-t+\xi)-\W(-\xi)$ which is also a Brownian motion.
\end{theorem}
\begin{proof}
	For any $t\geq\tau$, $r\leq0$ and $\Lambda\in\mathscr{J}$, by \eqref{ItoF}, we obtain
	\begin{align}\label{RLLE1}
	&\Lambda(\u(t;\tau+r,\theta_{-t}\omega,\u^*))-\Lambda(\u(\tau;\tau+r,\theta_{-t}\omega,\u^*))\nonumber\\&=\int_{\tau}^{t}\left(\Lambda'(\u(\xi;\tau+r,\theta_{-t}\omega,\u^*)),\sigma\u(\xi;\tau+r,\theta_{-t}\omega,\u^*)\right)\d\widetilde{\W}(\xi)\nonumber\\&\quad+\int_{\tau}^{t}\left\langle\Lambda'(\u(\xi;\tau+r,\theta_{-t}\omega,\u^*)),L(\u(\xi;\tau+r,\theta_{-t}\omega,\u^*),\xi)\right\rangle\d \xi\nonumber\\&\quad+\frac{1}{2}\int_{\tau}^{t}\left(\Lambda''(\u(\xi;\tau+r,\theta_{-t}\omega,\u^*))(\sigma\u(\xi;\tau+r,\theta_{-t}\omega,\u^*))(\sigma\u(\xi;\tau+r,\theta_{-t}\omega,\u^*))\right)\d\xi,
	\end{align}
where $\widetilde{\W}(r)=\W(-t+r)-\W(-t)$. Since $t\geq\tau$, we infer from \eqref{Phi1}, \eqref{ISM1} and \eqref{RLLE1} that 
	\begin{align}\label{RLLE2}
	&\int_{\mathscr{A}(t,\omega)}\Lambda(\u^*)\uprho_{t,\omega}(\d\u^*)-\int_{\mathscr{A}(\tau,\theta_{-(t-\tau)}\omega)}\Lambda(\u^*)\uprho_{\tau,\theta_{-(t-\tau)}\omega}(\d\u^*)\nonumber\\&=\underset{s\to-\infty}{\rm LIM}\frac{1}{-s}\int_{s}^{0}\int_{\H}\Lambda\big(\Phi(t-\tau-r,\tau+r,\theta_{-(t-\tau-r)}\omega,\u^*)\big)\uprho_{\tau+r,\theta_{-(t-\tau-r)}\omega}(\d\u^*)\d r\nonumber\\&\quad -\underset{s\to-\infty}{\rm LIM}\frac{1}{-s}\int_{s}^{0}\int_{\H}\Lambda\big(\Phi(-r,\tau+r,\theta_{-(t-\tau-r)}\omega,\u^*)\big)\uprho_{\tau+r,\theta_{-(t-\tau-r)}\omega}(\d\u^*)\d r\nonumber\\&=\underset{s\to-\infty}{\rm LIM}\frac{1}{-s}\int_{s}^{0}\int_{\H}\Lambda\big(\u(t;\tau+r,\theta_{-t}\omega,\u^*)\big)\uprho_{\tau+r,\theta_{-(t-\tau-r)}\omega}(\d\u^*)\d r\nonumber\\&\quad -\underset{s\to-\infty}{\rm LIM}\frac{1}{-s}\int_{s}^{0}\int_{\H}\Lambda\big(\u(\tau;\tau+r,\theta_{-t}\omega,\u^*)\big)\uprho_{\tau+r,\theta_{-(t-\tau-r)}\omega}(\d\u^*)\d r \nonumber\\&=\underset{s\to-\infty}{\rm LIM}\frac{1}{-s}\int_{s}^{0}\int_{\H}\int_{\tau}^{t}\big(\Lambda'(\u(\xi;\tau+r,\theta_{-t}\omega,\u^*)),\nonumber\\&\qquad\qquad\qquad\qquad\qquad\quad\sigma\u(\xi;\tau+r,\theta_{-t}\omega,\u^*)\big)\d\widetilde{\W}(\xi)\uprho_{\tau+r,\theta_{-(t-\tau-r)}\omega}(\d\u^*)\d r\nonumber\\&\quad+\underset{s\to-\infty}{\rm LIM}\frac{1}{-s}\int_{s}^{0}\int_{\H}\int_{\tau}^{t}\big\langle\Lambda'(\u(\xi;\tau+r,\theta_{-t}\omega,\u^*)),\nonumber\\&\qquad\qquad\qquad\qquad\qquad\quad L(\u(\xi;\tau+r,\theta_{-t}\omega,\u^*),\xi)\big\rangle\d \xi\uprho_{\tau+r,\theta_{-(t-\tau-r)}\omega}(\d\u^*)\d r\nonumber\\&\quad+\frac{1}{2}\underset{s\to-\infty}{\rm LIM}\frac{1}{-s}\int_{s}^{0}\int_{\H}\int_{\tau}^{t}\big(\Lambda''(\u(\xi;\tau+r,\theta_{-t}\omega,\u^*))(\sigma\u(\xi;\tau+r,\theta_{-t}\omega,\u^*))\nonumber\\&\qquad\qquad\qquad\qquad\qquad\quad(\sigma\u(\xi;\tau+r,\theta_{-t}\omega,\u^*))\big)\d\xi\uprho_{\tau+r,\theta_{-(t-\tau-r)}\omega}(\d\u^*)\d r \nonumber\\&=\underset{s\to-\infty}{\rm LIM}\frac{1}{-s}\int_{s}^{0}\int_{\tau}^{t}\int_{\H}\big(\Lambda'(\u(\xi;\tau+r,\theta_{-t}\omega,\u^*)),\nonumber\\&\qquad\qquad\qquad\qquad\qquad\quad\sigma\u(\xi;\tau+r,\theta_{-t}\omega,\u^*)\big)\uprho_{\tau+r,\theta_{-(t-\tau-r)}\omega}(\d\u^*)\d\widetilde{\W}(\xi)\d r\nonumber\\&\quad+\underset{s\to-\infty}{\rm LIM}\frac{1}{-s}\int_{s}^{0}\int_{\tau}^{t}\int_{\H}\big\langle\Lambda'(\u(\xi;\tau+r,\theta_{-t}\omega,\u^*)),\nonumber\\&\qquad\qquad\qquad\qquad\qquad\quad L(\u(\xi;\tau+r,\theta_{-t}\omega,\u^*),\xi)\big\rangle\uprho_{\tau+r,\theta_{-(t-\tau-r)}\omega}(\d\u^*)\d \xi\d r\nonumber\\&\quad+\frac{1}{2}\underset{s\to-\infty}{\rm LIM}\frac{1}{-s}\int_{s}^{0}\int_{\tau}^{t}\int_{\H}\big(\Lambda''(\u(\xi;\tau+r,\theta_{-t}\omega,\u^*))(\sigma\u(\xi;\tau+r,\theta_{-t}\omega,\u^*))\nonumber\\&\qquad\qquad\qquad\qquad\qquad\quad(\sigma\u(\xi;\tau+r,\theta_{-t}\omega,\u^*))\big)\uprho_{\tau+r,\theta_{-(t-\tau-r)}\omega}(\d\u^*)\d\xi\d r.
\end{align}
It implies from \eqref{Phi1} and the cocycle property of $\Phi$ that
\begin{align*}
	\u(\xi;\tau+r,\theta_{-t}\omega,\u^*)&=\Phi(\xi-\tau-r,\tau+r,\theta_{-(t-\tau-r)}\omega,\u^*)\nonumber\\&=\Phi(\xi-\tau,\tau,\theta_{-(t-\tau)}\omega,\Phi(-r,\tau+r,\theta_{-(t-\tau-r)}\omega,\u^*)),
\end{align*}
which, along with \eqref{RLLE2}, Lemma \ref{G-Con} and Definition \ref{D-ISM} (definition of invariant sample measures), follows that
\begin{align}
&\int_{\mathscr{A}(t,\omega)}\Lambda(\u^*)\uprho_{t,\omega}(\d\u^*)-\int_{\mathscr{A}(\tau,\theta_{-(t-\tau)}\omega)}\Lambda(\u^*)\uprho_{\tau,\theta_{-(t-\tau)}\omega}(\d\u^*)\nonumber\\&= \underset{s\to-\infty}{\rm LIM}\frac{1}{-s}\int_{s}^{0}\int_{\tau}^{t}\int_{\H}\big(\Lambda'(\Phi(\xi-\tau,\tau,\theta_{-(t-\tau)}\omega,\u^*)),\nonumber\\&\qquad\qquad\qquad\qquad\qquad\quad\sigma\Phi(\xi-\tau,\tau,\theta_{-(t-\tau)}\omega,\u^*)\big)\uprho_{\tau,\theta_{-(t-\tau)}\omega}(\d\u^*)\d\widetilde{\W}(\xi)\d r\nonumber\\&\quad+\underset{s\to-\infty}{\rm LIM}\frac{1}{-s}\int_{s}^{0}\int_{\tau}^{t}\int_{\H}\big\langle\Lambda'(\Phi(\xi-\tau,\tau,\theta_{-(t-\tau)}\omega,\u^*)),\nonumber\\&\qquad\qquad\qquad\qquad\qquad\quad L(\Phi(\xi-\tau,\tau,\theta_{-(t-\tau)}\omega,\u^*),\xi)\big\rangle\uprho_{\tau,\theta_{-(t-\tau)}\omega}(\d\u^*)\d \xi\d r\nonumber\\&\quad+\frac{1}{2}\underset{s\to-\infty}{\rm LIM}\frac{1}{-s}\int_{s}^{0}\int_{\tau}^{t}\int_{\H}\big(\Lambda''(\Phi(\xi-\tau,\tau,\theta_{-(t-\tau)}\omega,\u^*))(\sigma\u(\xi-\tau,\tau,\theta_{-(t-\tau)}\omega,\u^*))\nonumber\\&\qquad\qquad\qquad\qquad\qquad\quad(\sigma\u(\xi-\tau,\tau,\theta_{-{t-\tau}}\omega,\u^*))\big)\uprho_{\tau,\theta_{-(t-\tau)}\omega}(\d\u^*)\d\xi\d r  \nonumber\\&= \int_{\tau}^{t}\int_{\H}\big(\Lambda'(\Phi(\xi-\tau,\tau,\theta_{-(t-\tau)}\omega,\u^*)),\sigma\Phi(\xi-\tau,\tau,\theta_{-(t-\tau)}\omega,\u^*)\big)\uprho_{\tau,\theta_{-(t-\tau)}\omega}(\d\u^*)\d\widetilde{\W}(\xi)\nonumber\\&\quad+\int_{\tau}^{t}\int_{\H}\big\langle\Lambda'(\Phi(\xi-\tau,\tau,\theta_{-(t-\tau)}\omega,\u^*)), L(\Phi(\xi-\tau,\tau,\theta_{-(t-\tau)}\omega,\u^*),\xi)\big\rangle\uprho_{\tau,\theta_{-(t-\tau)}\omega}(\d\u^*)\d \xi\nonumber\\&\quad+\frac{1}{2}\int_{\tau}^{t}\int_{\H}\big(\Lambda''(\Phi(\xi-\tau,\tau,\theta_{-(t-\tau)}\omega,\u^*))(\sigma\u(\xi-\tau,\tau,\theta_{-(t-\tau)}\omega,\u^*))\nonumber\\&\qquad\qquad\qquad\qquad\qquad\quad(\sigma\u(\xi-\tau,\tau,\theta_{-{t-\tau}}\omega,\u^*))\big)\uprho_{\tau,\theta_{-(t-\tau)}\omega}(\d\u^*)\d\xi\nonumber\\&=\int_{\tau}^{t}\int_{\H}\left(\Lambda'(\u^*),\sigma\u^*\right)\uprho_{\xi,\theta_{-(t-\xi)}\omega}(\d\u^*)\d\widetilde{\W}(\xi)+\int_{\tau}^{t}\int_{\H}\left\langle\Lambda'(\u^*),L(\u^*,\xi)\right\rangle\uprho_{\xi,\theta_{-(t-\xi)}\omega}(\d\u^*)\d \xi \nonumber\\&\quad+\frac{1}{2}\int_{\tau}^{t}\int_{\H}\left(\Lambda''(\u^*)(\sigma\u^*)(\sigma\u^*)\right)\uprho_{\xi,\theta_{-(t-\xi)}\omega}(\d\u^*)\d \xi,
\end{align}
which completes the proof.
\end{proof}
\begin{remark}
	It is quite obvious from the calculations of Section \ref{sec3} that the system \eqref{CNSE} has $\mathfrak{D}$-pullback attractor $\widehat{\mathscr{A}}=\{\widehat{\mathscr{A}}(t,\omega):t\in\R,\omega\in\Omega\}\in\mathfrak{D}$, and there exists a family of Borel probability measures $\{\widehat{\uprho}_{t,\omega}\}_{t\in\R,\omega\in\Omega}$ such that $\uprho_{t,\omega}$ is an invariant sample measure of system \eqref{CNSE} supported on $\widehat{\mathscr{A}}(t,\omega)$. Moreover, by \cite[Theorem 5.1]{CY}, we have that the family of measures $\{\widehat{\uprho}_{t,\omega}\}_{t\in\R,\omega\in\Omega}$ satisfies a random Liouville type equation, that is, for all $t\geq \tau$ and $\omega\in\Omega$,
	\begin{align}\label{RLTT}
		&\int_{\widehat{\mathscr{A}}(t,\omega)}\Lambda(\v)\widehat{\uprho}_{t,\omega}(\d\v)-\int_{\widehat{\mathscr{A}}(\tau,\theta_{-(t-\tau)}\omega)}\Lambda(\v)\widehat{\uprho}_{\tau,\theta_{-(t-\tau)}\omega}(\d\v)\nonumber\\&=\int_{\tau}^{t}\int_{\widehat{\mathscr{A}}(r,\theta_{-(t-r)}\omega)}\langle \widehat{L}(\v,r,\theta_{-(t-r)}\omega),\Lambda'(\v)\rangle\widehat{\uprho}_{r,\theta_{-(t-r)}\omega}(\d\v)\d r,\ \ \text{ for all }\ \ \Lambda\in\mathscr{J},
	\end{align}
	where $\widehat{L}(\v,t,\theta_{t}\omega)=-\nu \A\v-\left[\frac{\sigma^2}{2}-\sigma y(\theta_{t}\omega)\right]\v-\z^{-1}(t,\omega)\B\big(\v\big)+\z(t,\omega)\mathcal{P}\f$.
\end{remark}
\begin{remark}\label{Sto-Liv}
	The result of above Theorem \ref{SLTT} can be seen as a stochastic Liouville theorem. If the stochastic statistical equilibrium has been attained by 2D SNSE, then the statistical information do not change with respect to time, that is, $\Lambda'(\u(\cdot;\cdot,\omega,\cdot))=0$. Therefore, it implies from \eqref{ISM2} and \eqref{SLTT1} that for each $\omega\in\Omega$,
	\begin{align}\label{RLT}
	\int_{\mathscr{A}(t,\omega)}\Lambda(\u)\uprho_{t,\omega}(\d\u)&=\int_{\mathscr{A}(\tau,\theta_{-(t-\tau)}\omega)}\Lambda(\Phi(t-\tau,\tau,\theta_{\tau-t}\omega,\u))\uprho_{\tau,\theta_{-(t-\tau)}\omega}(\d\u)\nonumber\\&=\int_{\mathscr{A}(\tau,\theta_{-(t-\tau)}\omega)}\Lambda(\u)\uprho_{\tau,\theta_{-(t-\tau)}\omega}(\d\u),
\end{align}
	$\tau\in\R$  and  $t\geq\tau$. The above  equality \eqref{RLT} says that the sample measures $\{\uprho_{t,\omega}\}_{t\in\R,\omega\in\Omega}$ are invariant under the action of the non-autonomous random DS $\Phi$. It declares that, for each $t\in\R$ and $\omega\in\Omega$, the shape of the random attractor $\mathscr{A}(t,\omega)$ could change randomly with the evolution of time from $\tau$ to $t$, along with the sample points $\omega\in\Omega$, but the measures of $\mathscr{A}(\tau,\theta_{-(t-\tau)}\omega)$ and $\mathscr{A}(t,\omega)$ would be the same. This is the stochastic (or random) version of the \emph{Liouville Theorem} in Statistical Mechanics. Therefore, we say that the invariant sample measures $\{\uprho_{t,\omega}\}_{t\in\R,\omega\in\Omega}$ of the 2D SNSE \eqref{SNSE} satisfy \emph{a stochastic Liouville type theorem}.
\end{remark}

\section{Invariant Measures and Ergodicity}\label{sec5}\setcounter{equation}{0}
In this section, we demonstrate the existence and uniqueness of invariant measures and ergodicity for random DS $\Phi$ associated with the system \eqref{SNSE}. Since  we are applying the abstract theory established in the work \cite{CF}, we restrict our deterministic forcing term to be independent of time $t$. From now onward, $\f$ is  time-independent and $\f\in\L^2(\R^2)$ (autonomous case).

The authors  in the work \cite{CF} proved that if a random DS $\Phi$ has a compact invariant random set, then there exists an invariant measure for $\Phi$ (cf. \cite[Corollaries 4.4 and 4.6]{CF}). Hence, the existence of invariant measures for the 2D autonomous SNSE \eqref{SNSE} is a direct consequence of \cite[Corollaries 4.4 and 4.6]{CF} and Theorem \ref{MainTheoRA}, because the random attractor itself is a \emph{compact invariant random set}. 

Define two $\sigma$-algebras corresponding to the past and future, respectively, by
{\begin{align*}
	\mathscr{F}^-=\sigma\{\omega\mapsto\Phi(t,\theta_{-\tau}\omega): 0\leq t\leq \tau\}
\text{ and }
	\mathscr{F}^+=\sigma\{\omega\mapsto\Phi(t,\theta_{\tau}\omega): t,\tau\geq0\}.
\end{align*}}
\begin{theorem}[Corollary 4.6, \cite{CF}]\label{CFMT}
	Let $\omega\mapsto\mathscr{A}(\omega)$ be a $\Phi$-invariant compact set which is measurable with respect to the past $\mathscr{F}^-$ for a random DS $\Phi$ and $\Phi$ is a random DS  whose one-point motions form a Markov family, and such that $\mathscr{F}^+$ and $\mathscr{F}^-$ are independent. Then there exists an invariant measure $\upmu$ for the associated Markov semigroup. Furthermore, the limit
	\begin{align*}
		\uprho_{\omega}=\lim_{t\to\infty}\Phi(t,\theta_{-t}\omega)\upmu
	\end{align*}
exists $\mathbb{P}$-a.s., $\upmu=\int_{\Omega}\uprho_{\omega}\d\mathbb{P}(\omega)=\mathbb{E}[\uprho_{\cdot}]$, and $\uprho$ is a Markov measure.
\end{theorem}
\subsection{Existence}
Since the deterministic forcing term in the system \eqref{SNSE} is time independent, the non-autonomous random DS will reduce to autonomous random DS. Therefore, the random DS $\Phi:\R^+\times\Omega\times\H\to\H$, for $t\geq\tau$, $\omega\in\Omega$ and $\u_0\in\H$, is defined by
\begin{align}\label{Phi2}
	\Phi(t-\tau,\theta_{\tau}\omega,\u_{\tau}) =\u(t;\tau,\omega,\u_{\tau})=\frac{\v(t;\tau,\omega,\v_{\tau})}{\z(t,\omega)} \ \text{ with }\  \v_{\tau}=\z(\tau,\omega)\u_{\tau}.
\end{align}
For a Banach space $\X$, let $\mathscr{B}_b(\X)$ be the space of all bounded and Borel measurable functions on $\X$, and $\C_b(\X)$ be the space of all bounded and continuous functions on $\X$. Let us define the transition operator $\{\mathrm{T}_t\}_{t\geq 0}$ by 
\begin{align}\label{71}
	\mathrm{T}_t g(\x)=\int_{\Omega}g(\Phi(t,\omega,\x))\d\mathbb{P}(\omega)=\E\left[g(\Phi(t,\cdot,\x))\right],
\end{align}
for all $g\in\mathscr{B}_b(\H)$, where $\Phi$ is the random DS corresponding to the 2D SNSE \eqref{SNSE} defined by \eqref{Phi2}. Since $\Phi$ is continuous (Lemma \ref{Continuity}), the following result holds due to \cite[Proposition 3.8]{BL}. 
\begin{lemma}\label{Feller}
	The family $\{\mathrm{T}_t\}_{t\geq 0}$ is Feller, that is, $\mathrm{T}_tg\in\C_{b}(\H)$ if $g\in\C_b(\H)$. Moreover, for any $g\in\C_b(\H)$, $\mathrm{T}_tg(\x)\to g(\x)$ as $t\downarrow 0$. 
\end{lemma}
\begin{definition}
	A Borel probability measure $\mu$ on $\H$  is called an \emph{invariant measure} for a Markov semigroup $\{\mathrm{T}_t\}_{t\geq 0}$ of Feller operators on $\C_b(\H)$ if and only if $$\mathrm{T}_{t}^*\mu=\mu, \ t\geq 0,$$ where $(\mathrm{T}_{t}^*\mu)(\Gamma)=\int_{\H}\mathrm{P}_{t}(\y,\Gamma)\mu(\d\y),$ for $\Gamma\in\mathscr{B}(\H)$ and  $\mathrm{T}_t(\y,\cdot)$ is the transition probability, $\mathrm{T}_{t}(\y,\Gamma)=\mathrm{T}_{t}(\chi_{\Gamma})(\y),\ \y\in\H$.
\end{definition}
One can establish that $\Phi$ is a Markov random DS (cf. \cite[Theorem 5.6]{CF}), that is, $\mathrm{T}_{s_1+s_2}=\mathrm{T}_{s_1}\mathrm{T}_{s_2}$, for all $s_1,s_2\geq 0$. 	It is known by Theorem \ref{CFMT} that there exists a Feller invariant probability measure $\mu$ for a Markov random DS $\Phi$ which has an invariant compact random set on a Polish space. Hence we have the following result due to Theorems \ref{MainTheoRA} and \ref{CFMT}.

\begin{theorem}\label{thm6.3}
	Assume that $\f\in\L^2(\R^2)$. Then, the Markov semigroup $\{\mathrm{T}_t\}_{t\geq 0}$ induced by the flow $\Phi$ on $\H$ has an invariant measure $\upmu$. The associated flow-invariant Markov measure $\uprho$ on $\Omega\times\H$ has the property that its disintegration $\omega\mapsto\uprho_{\omega}$ is supported by the attractor $\mathscr{A}(\omega)$, where $\mathscr{A}(\omega)$ is a random attractor for autonomous SNSE \eqref{1}.
\end{theorem}

\begin{remark}\label{4.4}
	Theorem \ref{thm6.3}  establishes the existence of invariant measures in $\H$. Since we have proved the existence of a unique $(\H,\V)$-pullback $\mathfrak{D}$-random attractor, it can be shown in a similar way (same as in the case of $\H$) that there exists an invariant measure for the 2D SNSE \eqref{SNSE} in $\V$ also.
\end{remark}

\subsection{Uniqueness} In this section, we prove the uniqueness of invariant measures for the system \eqref{SNSE}  by using the linear structure of multiplicative noise and the exponential stability of solutions. For this purpose, we  consider the external forcing $\f=\textbf{0}$ in the system \eqref{SNSE}.

\begin{lemma}\label{ExpoStability}
Assume that \emph{$\f=\textbf{0}$}. Then, there exists $T(\omega)>0$ such that the solution of the system \eqref{SNSE} satisfies the following exponential estimate:
\begin{align}\label{U1}
	\|\u_1(t)-\u_2(t)\|^2_{\H}\leq e^{-\frac{\sigma^2}{4}t}\|\u_{1,0}-\u_{2,0}\|^2_{\H},  \ \  \text{ for all }\  t\geq T(\omega),
\end{align}
where $\u_1(t):=\u_{1}(t;0,\omega,\u_{1,0})$ and $\u_{2}(t):=\u_{2}(t;0,\omega,\u_{2,0})$ be two solutions of the system \eqref{CNSE} with respect to the initial data $\u_{1,0}$ and $\u_{2,0}$ at $0$, respectively. 
\end{lemma}
\begin{proof}
	Let  $\v_1(t):=z(t,\omega)\u_{1}(t)$ and $\v_{2}(t):=z(t,\omega)\u_{2}(t)$ be two solutions of the system \eqref{CNSE} with respect to the initial data $\v_{1,0}=z(0,\omega)\u_{1,0}$ and $\v_{2,0}=z(0,\omega)\u_{2,0}$ at $0$, respectively. Then $\mathfrak{X}(\cdot)=\v_{1}(\cdot)-\v_{2}(\cdot)$ with $\mathfrak{X}(0)=\v_{1,0}-\v_{2,0}$ satisfies
	\begin{align}\label{ES1}
		\frac{\d\mathfrak{X}(t)}{\d t}+\nu \A\mathfrak{X}(t)+\left[\frac{\sigma^2}{2}-\sigma y(\theta_{t}\omega)\right]\mathfrak{X}(t)=-\z^{-1}(t,\omega)\left\{\B\big(\v_1(t)\big)-\B\big(\v_2(t)\big)\right\},
	\end{align}
	for a.e. $t\geq0$ in $\V^*$. Multiplying \eqref{ES1} with $\mathfrak{X}(\cdot)$ and then integrating over $\R^2$, we infer
	\begin{align}\label{ES2}
		&\frac{1}{2}\frac{\d}{\d t} \|\mathfrak{X}(t)\|^2_{\H} +\nu \|\nabla\mathfrak{X}(t)\|^2_{\H} + \left[\frac{\sigma^2}{2}-\sigma y(\theta_{t}\omega)\right]\|\mathfrak{X}(t)\|^2_{\H} \nonumber\\&=-\z^{-1}(t,\omega)\left\langle\B\big(\v_1(t)\big)-\B\big(\v_2(t)\big), \mathfrak{X}(t)\right\rangle,
	\end{align}
	for a.e. $t\geq0$. Using \eqref{b0}, \eqref{441}, H\"older's, Ladyzhenskaya's and Young's inequalities, we obtain
	\begin{align}\label{ES3}
		\left| \z^{-1}(t,\omega)\left\langle\B\big(\v_1\big)-\B\big(\v_2\big), \mathfrak{X}\right\rangle\right|&=\left|\z^{-1}(t,\omega)b(\mathfrak{X},\mathfrak{X},\v_2)\right|\nonumber\\&\leq 2^{\frac{1}{2}}\z^{-1}(t,\omega)\|\mathfrak{X}\|_{\H}\|\nabla\mathfrak{X}\|_{\H}\|\nabla\v_2\|_{\H}\nonumber\\&\leq\frac{\nu}{2}\|\nabla\mathfrak{X}\|^2_{\H}+\frac{\z^{-2}(t,\omega)}{\nu}\|\nabla\v_2\|^2_{\H}\|\mathfrak{X}\|^2_{\H}.
	\end{align}
	Making use of \eqref{ES3} in \eqref{ES2}, we infer
	\begin{align*}
		&	\frac{\d}{\d t} \|\mathfrak{X}(t)\|^2_{\H} +\left[\sigma^2-2\sigma y(\theta_{t}\omega)-\frac{2\z^{-2}(t,\omega)}{\nu}\|\nabla\v_2\|^2_{\H}\right]\|\mathfrak{X}(t)\|^2_{\H} \leq 0,
	\end{align*}
 for a.e. $t\geq0$ and for all $\omega\in\Omega$,  and an application of variation of constant formula implies 
	\begin{align*}
		\|\mathfrak{X}(t)\|^2_{\H}&\leq \exp\left\{-\sigma^2t +2\sigma \int_{0}^{t}y(\theta_{r}\omega)\d r+\frac{2}{\nu}\int_{0}^{t}e^{2\sigma y(\theta_{r}\omega)}\|\nabla\v_2(r)\|^2_{\H}\d r \right\}\|\mathfrak{X}(0)\|^2_{\H},
	\end{align*} 
	for all $t\geq0$ and $\omega\in\Omega$. Using the transformation \eqref{Phi2} and the properties \eqref{Z3} of the continuous process $y(\cdot)$, we obtain that there exists a time $T_1(\omega)>0$ such that 
		\begin{align}\label{ES6}
		&\|\u_1(t)-\u_2(t)\|^2_{\H}\nonumber\\&\leq \exp\left\{-\sigma^2t+2\sigma y(\theta_{t}\omega) +2\sigma \int_{0}^{t}y(\theta_{r}\omega)\d r+\frac{2}{\nu}\int_{0}^{t}e^{2\sigma y(\theta_{r}\omega)}\|\nabla\v_2(r)\|^2_{\H}\d r-2\sigma y(\omega) \right\}\nonumber\\&\quad\times\|\u_{1,0}-\u_{2,0}\|^2_{\H}\nonumber\\&\leq \exp\left\{-\frac{\sigma^2}{2}t+\frac{2}{\nu}\int_{0}^{t}e^{2\sigma y(\theta_{r}\omega)}\|\nabla\v_2(r)\|^2_{\H}\d r-2\sigma y(\omega) \right\} \|\u_{1,0}-\u_{2,0}\|^2_{\H}, 
	\end{align} 
for all $t\geq T_1(\omega)$ and $\omega\in\Omega$. From the first equation of the system \eqref{CNSE}, \eqref{b0} and Young's inequality, we obtain
\begin{align*}
	&\frac{\d}{\d t} \|\v_2(t)\|^2_{\H}+ \left[\frac{\sigma^2}{2}-2 \sigma y(\theta_{t}\omega)\right]\|\v_2(t)\|^2_{\H} +2\nu\|\nabla\v_2(t)\|^2_{\H} \leq 0,
\end{align*}
which gives (by variation of constant formula)
\begin{align*}
		& \|\v_2(t)\|^2_{\H}e^{\frac{\sigma^2}{2}t-2\sigma\int_{0}^{t}y(\theta_{r}\omega)\d r}+2\nu \int_{0}^{t}e^{\frac{\sigma^2}{2}s-2\sigma\int_{0}^{s}y(\theta_{r}\omega)\d r}\|\nabla\v_2(s)\|^2_{\H}\d s \leq \|\v_{2,0}\|^2_{\H},\nonumber\\
		&\Rightarrow 
		2\nu \int_{\xi}^{t}e^{\frac{\sigma^2}{2}s-2\sigma\int_{0}^{s}y(\theta_{r}\omega)\d r}\|\nabla\v_2(s)\|^2_{\H}\d s\leq \|\v_{2,0}\|^2_{\H},
\end{align*}
for all $t\geq\xi\geq0$ and $\omega\in\Omega$.  By the ergodicity property \eqref{Z3} of the continuous process $y(\cdot)$, we find $T_2(\omega)>0$ such that 
\begin{align}\label{ES7}
	\int_{T_2}^{t}e^{\frac{\sigma^2}{4}s}\|\nabla\v_2(s)\|^2_{\H}\d s\leq \int_{T_2}^{t}e^{\frac{\sigma^2}{2}s-2\sigma\int_{0}^{s}y(\theta_{r}\omega)\d r}\|\nabla\v_2(s)\|^2_{\H}\d s& \leq \frac{1}{2\nu}\|\v_{2,0}\|^2_{\H}, \  \ \mbox{ for all $t\geq T_2(\omega)$.}
\end{align}
From $\lim\limits_{t\to \pm \infty} \frac{|y(\theta_t\omega)|}{|t|}=0$, there exists a time $T_3(\omega)\geq T_2(\omega)$ such that 
\begin{align}\label{ES8}
	&\int_{0}^{t}e^{2\sigma y(\theta_{r}\omega)}\|\nabla\v_2(r)\|^2_{\H}\d r\nonumber\\&=\int_{0}^{T_3}e^{2\sigma y(\theta_{r}\omega)}\|\nabla\v_2(r)\|^2_{\H}\d r+\int_{T_3}^{t}e^{2\sigma y(\theta_{r}\omega)}\|\nabla\v_2(r)\|^2_{\H}\d r\nonumber\\&\leq \sup_{r\in[0,T_3]}\left[e^{2\sigma y(\theta_{r}\omega)}\right]\int_{0}^{T_3}\|\nabla\v_2(r)\|^2_{\H}\d r +\int_{T_3}^{t}e^{\frac{\sigma^2}{4}r}\|\nabla\v_2(r)\|^2_{\H}\d r\nonumber\\&\leq \sup_{r\in[0,T_3]}\left[e^{2\sigma y(\theta_{r}\omega)}\right]\int_{0}^{T_3}\|\nabla\v_2(r)\|^2_{\H}\d r +\int_{T_2}^{t}e^{\frac{\sigma^2}{4}r}\|\nabla\v_2(r)\|^2_{\H}\d r\nonumber\\&\leq \frac{2}{\sigma^2} \sup_{r\in[0,T_3]}\left[e^{2\sigma y(\theta_{r}\omega)}\right]\rho(T_3,\omega,\v_{2,0},\f) +\frac{1}{2\nu}\|\v_{2,0}\|^2_{\H}:=\rho^*(T_3,\omega,\v_{2,0},\f),
\end{align}
for all $t\geq T_3(\omega)$ and $\omega\in\Omega$, where we have used \eqref{ei2} and \eqref{ES7} in the final inequality. Let us choose $T_4(\omega)\geq\max\{T_1(\omega),T_3(\omega)\}$, then \eqref{ES6} and \eqref{ES8} imply
\begin{align}\label{ES9}
		&\|\u_1(t)-\u_2(t)\|^2_{\H}\leq \exp\left\{-\frac{\sigma^2}{2}t+\frac{2}{\nu}\rho^*(T_3,\omega,\v_{2,0},\f)-2\sigma y(\omega) \right\} \|\u_{1,0}-\u_{2,0}\|^2_{\H},
\end{align}
for all $t\geq T_4(\omega)$ and $\omega\in\Omega$. Now, there exists a time $T(\omega)\geq T_4(\omega)$ such that 
\begin{align}\label{ES10}
	\frac{2}{\nu}\rho^*(T_3,\omega,\v_{2,0},\f)-2\sigma y(\omega) \leq\frac{\sigma^2}{4}t,
\end{align}
for all $t\geq T(\omega)$ and $\omega\in\Omega$. By \eqref{ES9}-\eqref{ES10}, we have
\begin{align*}
	\|\u_1(t)-\u_2(t)\|^2_{\H}\leq e^{-\frac{\sigma^2}{4}t}\|\u_{1,0}-\u_{2,0}\|^2_{\H},  \ \ \ \text{ for all } t\geq T(\omega)\ \text{ and }\ \omega\in\Omega,
\end{align*}
which completes the proof.
\end{proof}

\begin{theorem}\label{UEIM}
	Assume that \emph{$\f=\textbf{0}$} and $\u_0\in\H$ be given. Then, there is a unique invariant measure to the system \eqref{SNSE} which is ergodic and strongly mixing.
\end{theorem}
\begin{proof}
	For $\psi\in \text{Lip}(\H)$ (Lipschitz $\psi$) and an invariant measure $\upmu$, we have for all $t\geq T(\omega)$,
	\begin{align*}
			\left|\mathrm{T}_t\psi(\u_0)-\int_{\H}\psi(\v_0)\upmu(\d \v_0)\right|&=	\left|\E[\psi(\u(t,\u_0))]-\int_{\H}\mathrm{T}_t\psi(\v_0)\upmu(\d \v_0)\right|\nonumber\\&=\left|\int_{\H}\E\left[\psi(\u(t,\u_0))-\psi(\u(t,\v_0))\right]\upmu(\d \v_0)\right|\nonumber\\&\leq L_{\psi}\int_{\H}\E\left[\left\|\u(t,\u_0)-\u(t,\v_0)\right\|_{\H}\right]\upmu(\d \v_0)\nonumber\\&\leq L_{\psi}\mathbb{E}\left[e^{-\frac{\sigma^2}{8}t}\right]\int_{\H}\left\|\u_0-\v_0\right\|_{\H}\upmu(\d \v_0)
		\to 0 \ \text{ as }\  t\to \infty,
	\end{align*}
	since $\int_{\H}\|\v_0\|_{\H}\upmu(\d \v_0)+\int_{\H}\|\u_0\|_{\H}\upmu(\d \v_0)<+\infty$. Hence, we conclude
	\begin{align}\label{U2}
		\lim_{t\to\infty}\mathrm{T}_t\psi(\u_0)=\int_{\H}\psi(\v_0)\d\upmu(\v_0), \ \upmu\text{-a.s., for all }\ \u_0\in\H\ \text{ and }\  \psi\in\C_b(\H),
	\end{align} 
	by the density of $\text{Lip}(\H)$ in $\C_b (\H)$. Since we have a stronger result that $\mathrm{T}_t\psi(\u_0)$ converges exponentially fast to the equilibrium, this property is known as the \emph{exponential mixing property}. Now suppose that  $\widetilde{\upmu}$ is an another invariant measure, then we have for all $t\geq T(\omega)$,
	\begin{align}
		&\left|\int_{\H}\psi(\u_0)\upmu(\d\u_0)-\int_{\H}\psi(\v_0)\widetilde{\upmu}(\d \v_0)\right|\nonumber\\&= \left|\int_{\H}\mathrm{T}_t\psi(\u_0)\upmu(\d \u_0)-\int_{\H}\mathrm{T}_t\psi(\v_0)\widetilde{\upmu}(\d \v_0)\right|\nonumber\\&=\left|\int_{\H}\int_{\H}\left[\mathrm{T}_t\psi(\u_0)-\mathrm{T}_t\psi(\v_0)\right]\upmu(\d \u_0)\wi\upmu(\d \v_0)\right|\nonumber\\&=\left|\int_{\H}\int_{\H}\E\left[\psi(\u(t,\u_0))-\psi(\u(t,\v_0))\right]\upmu(\d \u_0)\wi\upmu(\d \v_0)\right|\nonumber\\&\leq L_{\psi}\int_{\H}\int_{\H}\E\left[\left\|\u(t,\u_0)-\u(t,\v_0)\right\|_{\H}\right]\upmu(\d \u_0)\wi\upmu(\d \v_0) \nonumber\\&\leq L_{\psi}\mathbb{E}\left[e^{-\frac{\sigma^2}{8}t}\right] \int_{\H}\int_{\H}\E\left[\left\|\u_0-\v_0\right\|_{\H}\right]\upmu(\d \u_0)\wi\upmu(\d \v_0) \nonumber\\&\leq L_{\psi}\mathbb{E}\left[e^{-\frac{\sigma^2}{8}t}\right]\bigg(\int_{\H}\|\u_0\|_{\H}\upmu(\d \u_0)+\int_{\H}\|\v_0\|_{\H}\widetilde{\upmu}(\d \v_0)\bigg)\to 0 \ \text{ as }\ t\to\infty,
	\end{align}
	since $\int_{\H}\|\u_0\|_{\H}\upmu(\d \u_0)<+\infty$ and $\int_{\H}\|\v_0\|_{\H}\widetilde{\upmu}(\d \v_0)<+\infty$.  Since $\upmu$ is the unique invariant measure for $(\mathrm{T}_t)_{t\geq 0}$, it follows from \cite[Theorem 3.2.6]{GDJZ} that $\upmu$ is ergodic also. 
\end{proof}
\begin{remark}
	Since, for $\f=\mathbf{0}$, Theorem \ref{UEIM}  demonstrates the unique invariant measure in $\H$. As $\V\subset\H$ is a closed subspace, the uniqueness of invariant measure in $\V$ also follows from Theorem \ref{UEIM}. 
\end{remark}
\begin{remark}
	As zero solution is always a solution to the system \eqref{1} with zero initial data and $\f=\boldsymbol{0}$, it infers from Lemma \ref{ExpoStability} that the non-zero solutions to the system \eqref{1} are converging to $0$ exponentially fast when time goes to $\infty$. Particularly, the unique ergodic measure of the system \eqref{1} with $\f=\boldsymbol{0}$ is just the trivial one, that is, the Dirac delta measure centered at zero, which is obviously invariant since	the zero solution is the steady state of the system. Similarly, the pullback random attractor is just a singleton set $\{0\}$. 
\end{remark}

	\medskip\noindent
	{\bf Acknowledgments:} The first author would like to thank the Council of Scientific $\&$ Industrial Research (CSIR), India for financial assistance (File No. 09/143(0938)/2019-EMR-I).  M. T. Mohan would  like to thank the Department of Science and Technology (DST), Govt of India for Innovation in Science Pursuit for Inspired Research (INSPIRE) Faculty Award (IFA17-MA110). We would also like to thank the anonymous reviewer of the paper \cite{KM6} for raising a query regarding the use of It\^o noise in place of  Stratonovich noise,  which is  the first motivation of this paper.

	\medskip\noindent	{\bf  Declarations:} 
	
	\noindent 	{\bf  Ethical Approval:}   Not applicable 
	
	\noindent  {\bf   Competing interests: } The authors declare no competing interests. 
	
	\noindent 	{\bf   Authors' contributions: } All authors have contributed equally. 
	
	\noindent 	{\bf   Funding: } CSIR, India, File No. 09/143(0938)/2019-EMR-I (Kush Kinra), DST, India, IFA17-MA110 (M. T. Mohan)

	\noindent 	{\bf   Availability of data and materials: } Not applicable.

\end{document}